\documentclass[paper,onefignum,onetabnum]{siamart190516}

\usepackage{lipsum}
\usepackage{amsfonts}
\usepackage{graphicx}
\usepackage{epstopdf}
\usepackage{algorithmic}
\ifpdf
  \DeclareGraphicsExtensions{.eps,.pdf,.png,.jpg}
\else
  \DeclareGraphicsExtensions{.eps}
\fi

\newsiamremark{remark}{Remark}
\newsiamremark{hypothesis}{Hypothesis}
\crefname{hypothesis}{Hypothesis}{Hypotheses}
\newsiamthm{claim}{Claim}

\headers{Inventory Control with Unknown Demand Trend}{S. Federico, G. Ferrari, and N. Rodosthenous}

\title{TWO-SIDED SINGULAR CONTROL OF AN INVENTORY WITH UNKNOWN DEMAND TREND\thanks{%Submitted to the editors DATE.\funding{
Funded by the Deutsche Forschungsgemeinschaft (DFG, German Research Foundation) – SFB 1283/2 2021 – 317210226.}}%}

\author{Salvatore Federico\thanks{
Dipartimento di Economia, Universit\`a di Genova, Piazza F.\ Vivaldi 5, 16126, Genova, Italy 
(\email{salvatore.federico@unige.it}).} %, \url{http://www.imag.com/\string~ddoe/}).}
\and Giorgio Ferrari\thanks{
Center for Mathematical Economics (IMW), Bielefeld University, Universit\"atsstrasse 25, 33615, Bielefeld, Germany 
(\email{giorgio.ferrari@uni-bielefeld.de}).} %, \email{jesmith@fictional.edu}).}
\and Neofytos Rodosthenous\thanks{
Department of Mathematics, University College London, Gower St, London WC1E 6BT, UK (\email{n.rodosthenous@ucl.ac.uk}).} }

\usepackage{amsopn}

\ifpdf
\hypersetup{
  pdftitle={Inventory Control with Unknown Demand Trend},
  pdfauthor={S. Federico, G. Ferrari, and N. Rodosthenous}
}
\fi

\usepackage{verbatim}
\usepackage{scalerel}

\usepackage{dsfont}

\def \E{\mathsf{E}}

\def \P{\mathsf{P}}
\def \Q{\mathsf{Q}}
\def \R{\mathbb{R}}
\def \F{\mathbb{F}}

\def\d{\mathrm{d}}

\newcommand{\argmax}{\text{argmax}}

\definecolor{red}{rgb}{1.0,0.0,0.0}

\definecolor{blu}{rgb}{0.0,0.0,1.0}

\definecolor{gre}{rgb}{0.03,0.50,0.03}

\renewcommand{\hat}{\widehat}

\newtheorem{assumption}[theorem]{Assumption}

\newcommand{\cvd}{$\quad\Box $
                  \medskip}

\begin{document}

\maketitle

% REQUIRED
\begin{abstract}
We study the problem of optimally managing an inventory with unknown demand trend. Our formulation leads to a stochastic control problem under partial observation, in which a Brownian motion with non-observable drift can be singularly controlled in both an upward and downward direction. 
We first derive the equivalent separated problem under full information, with state-space components given by the Brownian motion and the filtering estimate of its unknown drift, and we then completely solve this latter problem. Our approach uses the transition amongst three different but equivalent problem formulations, links between two-dimensional bounded-variation stochastic control problems and games of optimal stopping, and probabilistic methods in combination with refined viscosity theory arguments. We show substantial regularity of (a transformed version of) the value function, we construct an optimal control rule, and we show that the free boundaries delineating (transformed) action and inaction regions are bounded globally Lipschitz continuous functions. To our knowledge this is the first time that such a problem has been solved in the literature.
\end{abstract}

% REQUIRED

% REQUIRED

\section{Introduction}

In this paper, we consider the optimal management of inventory when the demand is stochastic and partially observed. 
There exists an enormous literature on optimal inventory management (see, e.g.~\cite{Zipkin} for an overview and the significance of inventory control in operations and profitability of companies). 
The optimal singular/impulsive control literature of stochastic inventory systems has so far assumed that the dynamics of the inventory is fully known to decision makers, see e.g.~\cite{
Bath66, 
DY13,
DY13_2, 
HarrisonTak, 
HarrisonTay78, 
HYZ17, 
Taksar85, 
XZZ19,  
YYY20}, 
amongst many others.  
Some of the most celebrated results are the optimality of (constant) threshold strategies determining $(a)$ base-stock policies -- maintaining inventory above a fixed shortage level -- and $(b)$ restrictions on the size of inventory, in order to manage storage-related costs. In this paper, we generalise the existing literature on the singular control of inventories by assuming that the demand rate or the mean of the random demand for the product is unknown to decision makers. 
This can be relevant to companies operating in newly established markets or producing a novel good, for which there is limited knowledge about the demand trend. 
In particular, we will show how the aforementioned optimal strategies are no longer triggered by constant thresholds, but by functions of the decision maker's learning process of the unknown demand rate. 
We further note that the analysis and results in this paper can also contribute to applications way beyond the inventory management literature; for instance, to cash balance management problems (see, e.g.~\cite{EppenFama69}), when the drift of the cash process is unknown to managers.

\textbf{The model and general results.} We consider decision makers who can observe in real time the evolution of the level of a Brownian inventory system $S_t$, which represents the production minus the stochastic demand for the product at time $t$ 
(see \cite{HarrisonTak, Taksar85, YYY20}). 
The inventory has a ``net demand'' rate $\mu$, {\it unknown} to decision makers, and a stochastic part modelling the demand volatility. 
We assume that the random variable $\mu \in \{\mu_0, \mu_1\}$, for $\mu_0, \mu_1 \in \R$, and the decision makers' prior belief is $\pi:=\P(\mu=\mu_1)\in(0,1)$. 
This is continuously updated as new information is revealed according to the natural filtration $\mathcal{F}^S_t$ of $S$, 
and takes the form $\Pi_t:=\P(\mu=\mu_1\,|\,\mathcal{F}^S_t)$ according to standard filtering techniques (see \cite{LSbook} for a survey).
Decision makers can control the inventory via a bounded-variation process $P_t = P_t^+ - P_t^-$, where $P_t^\pm$ are increasing processes defining the total amount of increase/decrease of inventory up to time $t$. 
The controlled inventory level is therefore given by 
$X_t = x + \mu t + \eta B_t + P_t^+ - P_t^-$, for $\eta>0$ and all $t \geq 0$; positive values model the excess inventory, while the absolute value of negative $X$ models the backlog in production. 

Both levels of excess inventory and backorder bear (non-necessarily symmetric) holding and shortage costs per unit of time, modelled via a suitable convex function $C(X)$. 
High holding/storage costs 
for large $X$ could suggest unloading part of excess inventory (e.g.~start promotions, send to outlets, donate, ship to another facility, or destroy) at a cost $K^-$ proportional to unloaded volume $P^-$. On the other hand, high shortage costs due to undesirable low $X$ could suggest placing inventory replenishment orders at a cost $K^+$ proportional to the ordered volume $P^+$.
However, there is a trade off due to the costs $K^\pm$ of controlling the inventory $X$ to keep $C(X)$ at ``reasonable" levels.
The question we thus study is {\it ``What is the optimal inventory management strategy that minimises the total expected (discounted) future holding, shortage and control costs, when the demand rate is unknown?''}. 
We allow the rate of increase/reduction $\d P^\pm$ to be unbounded and have an instantaneous effect on $X$, hence the question is mathematically formulated as a bounded-variation stochastic control problem of a linearly controlled one-dimensional diffusion with the novelty of a random (non-observable) drift $\mu$. 

Indeed, we prove the existence of an optimal control strategy $P^{\star \pm}$ and characterise it via two boundary functions of the belief process $\Pi$, which split the space in three distinct but connected regions: 
$(a)$ An action region divided in the areas below or above the boundaries, so that when $X$ is relatively small or large, decision makers should increase or decrease $X$ via $P^{\star \pm}$, respectively, to bring $X$ inside the area between the two boundaries; 
and $(b)$ an intermediate waiting (inaction) region, which is precisely the area between the two boundaries.
We further prove the monotonicity of these boundaries and completely characterise them in terms of monotone Lipschitz continuous curves solving a system of nonlinear integral equations.
To the best of our knowledge, the study and characterisation of the boundaries defining the solution of a bounded-variation stochastic control problem under partial information on the underlying diffusion dynamics, has never been addressed in the literature.

\textbf{Our contributions, approach and overview of mathematical analysis.} Our contribution in this paper is twofold.  
From the point of view of its application, even though the literature on the optimal management of inventory is extremely rich, as already discussed, there is no model where the demand is assumed to be partially observed and lump-sum as well as singularly continuous actions on the inventory are allowed. 
From the mathematical theory perspective, the literature on the optimal policy characterisation in singular stochastic control problems with partial observation is limited, and actually deals only with monotone controls 
\cite{CCF20, DeA20, DecampsVill, OksSul12}. 
On the contrary, we allow the decision maker to both decrease and increase the underlying process by using controls of bounded-variation. 
Our paper thus provides a first example where partial observation features have been considered in the setting of a bounded-variation control problem. By combining the well-established connection to Dynkin games, probabilistic methods of free-boundary theory and refined viscosity theory arguments, we present a methodology that allows to achieve the necessary regularity of the value function, leading to a characterisation of the optimal control rule.
This is our second main contribution, on which we elaborate in the remaining of this section. 
Note that, other scenarios of partial information on the drift, considered for  
investment timing \cite{DMV05}, 
asset trading \cite{DG12}, 
optimal liquidation  \cite{EkstromVai}, 
contract theory \cite{DeMS16}, lead to different mathematical formulations.

By relying on classical filtering theory (see \cite{LSbook}), we first derive the equivalent Markovian ``separated problem'', which is a genuine two-dimensional bounded-variation singular stochastic control problem $V$ with diffusive state-space dynamics $(X, \Pi)$. 
The traditional ``guess and verify'' approach is not effective, since the associated variational formulation involves partial differential equations (PDEs) with (gradient) boundary conditions, whose explicit solutions are not possible in general. 
We instead use a more direct approach that allows for a thorough study of the value function $V$'s regularity and structure, eventually leading to the optimal control strategy's characterisation. 

Via changes of coordinates we first transform the original controlled process $(X,\Pi)$ into $(X,\Phi)$ with (degenerate) decoupled dynamics and later into $(X,Y)$ for the problem's intrinsic parabolic formulation (see also~\cite{DeA20,JP17}). 
We connect our resulting two-dimensional bounded-variation stochastic control problems, under each formulation, to suitable zero-sum optimal stopping (Dynkin) games with two-dimensional, uncontrolled dynamics. 
We manage to characterise each games' optimal stopping strategies via {\it interlinked} pairs of monotone and bounded free boundary functions $a_\pm(\pi), b_\pm(\varphi)$ and $c_\pm(y)$, respectively. 
By using our probabilistic methodology in combination with viscosity theory arguments\footnote{It is worth noticing that the combination of viscosity arguments and probabilistic techniques of free-boundary problems have been already employed for the study of bounded-variation control problems in \cite{Federico2014}, \cite{FeFeSch-1} and \cite{FeFeSch-2}. However, in those papers the dynamic programming equation takes the form of a parameter-dependent ODE with gradient constraints, while in our paper it is a degenerate PDE with gradient constraints.} and switching between these three equivalent formulations: 
$(a)$ we achieve the notable $C^1$-global regularity of the transformed value function $\overline{V}(x,\varphi)$, and we deduce that its version $\widehat{V}(x,y)$ is actually such that $\widehat{V} \in C^{1}(\R^2;\R)$ and $\widehat{V}_{xx}$ is bounded in its relative continuation region; 
$(b)$ we use these properties in order to construct an optimal control strategy in terms of the likelihood ratio-dependent process $t \mapsto b_\pm (\Phi_t)$ according to a Skorokhod reflection; 
$(c)$ we obtain global Lipschitz continuity of the free boundaries $c_\pm(y)$, employed to show the global $C^1$-regularity of the Dynkin game's value $\widehat v(x,y)$ and obtain a system of nonlinear integral equations solved by $c_\pm$.
It is worth observing that backtracking the involved change of variables, the characterisation of $c_\pm$ effectively turns into a characterisation of $b_{\pm}$ defining the optimal control policy (and consequently of $a_\pm$ in the original $(x,\pi)$--coordinates). 

The Lipschitz regularity result is of particular independent interest, given its importance in obstacle problems (see the introduction of \cite{DeAStabile} for a detailed account on this and its related literature). 
The simple argument of our proof, exploiting the geometry of the $(x,\varphi)$-plane and the particular structure of its transformation into the $(x,y)$-plane, provides a method -- alternative to the more technical approach developed in \cite{DeAStabile} -- for obtaining the Lipschitz regularity of the optimal stopping boundaries.

Finally, note that by using our methodology, we manage to obtain the minimal (necessary) regularity in order to construct an optimal control strategy and verify its optimality. 
As in multi-dimensional singular stochastic control settings proving regularity properties of the control value function can be very challenging, having a methodology that takes a different route by effectively combining various techniques, can be helpful in studying other problems with similar structure.

\textbf{Structure of the paper.} The rest of this paper is organised as follows. 
In Section \ref{sec:model}, we present the model, formulate the control problem, and derive the separated problem $V$. 
In Section \ref{sec:OSgame}, we derive the first related optimal stopping game.  
Section \ref{sec:2transf} introduces the first useful change of coordinates. 
Section \ref{sec:HJB} then studies the regularity of the control problem's (transformed) value function $\overline{V}$. 
Section \ref{sec:Verif} presents the verification theorem and construction of an optimal control. 
Finally, in Section \ref{sec:FBchar}, we: 
introduce the last change of variables; 
obtain the Lipschitz-continuity of the corresponding (transformed) free boundaries $c_\pm$; 
prove the smooth-fit property of the transformed Dynkin game's value function $\widehat{v}$; 
and derive the integral equations for $c_\pm$.

\section{Problem Formulation and the Separated Problem} 
\label{sec:model}

On a complete probability space $(\Omega,\mathcal{F},\P)$, we define a one-dimensional Brownian motion $(B_t)_{t\geq 0}$ whose $\P$-augmented natural filtration is denoted by $(\mathcal{F}_t^B)_{t\geq 0}$. Moreover, we define a random variable $\mu$ which is independent of the Brownian motion $B$ and can take two possible real values, namely $\mu\in \{\mu_0,\mu_1\}$, where $\mu_0,\mu_1\in\R$. Without loss of generality, we assume henceforth that $\mu_1>\mu_0$ and that 
$\pi:=\P(\mu=\mu_1)\in(0,1).$

In absence of any intervention, the underlying (stochastic inventory) process $S_t$ as observed by the decision maker, follows the dynamics 
$\d S_t=\mu \d t+\eta \d B_t$, with $S_0=x \in \R$, 
for some $\eta>0$. Recall that the drift $\mu$ of the process $S$ is not observable by the decision maker, who can only monitor the evolution of the process $S$ itself. In light of this observation, the decision maker select their control strategy $P$ based solely on their observation of the process $S$. 
By denoting the natural filtration of any process $Y$ by $\mathbb{F}^Y:=(\mathcal{F}_t^Y)_{t\geq 0}$, we can therefore define the set of admissible controls
\begin{eqnarray*}
\mathcal{A}&:=&\{P:\Omega\times\R^+\to\R \ \mbox{such that} \ t\mapsto P_t \ \mbox{is right-continuous, (locally) of bounded } \\ \nonumber 
&& \mbox{variation and}\,\, P \ \mbox{is}  \ \F^S-\mbox{adapted}\}.
\end{eqnarray*}
To be more precise, we consider the minimal decomposition of the bounded-variation control $P \in \mathcal{A}$ to be 
$P_t=P_t^+ - P_t^-,$ 
where $P^+$ and $P^-$ are then nondecreasing, right-continuous $\F^S$--adapted processes. From now on, we set $P^{\pm}_{0-}=0$ a.s.\ for any $P \in \mathcal{A}$.
Hence, the reference (controlled inventory) process is given by 
\begin{equation*}
X_t^P := S_t + P_t = x+ \mu t + \eta B_t + P_t, 
\qquad \text{where } \, P\in\mathcal{A}.
\end{equation*}
Note that, the uncontrolled inventory process ($P \equiv 0$) takes the form $X^0 = S$. 

Given the aforementioned setting, the decision maker's goal is to minimise the overall (discounted) cost of holding, shortage and controlling the inventory process. 
In mathematical terms, the bounded-variation control problem of the decision maker is given by 
\begin{equation}
\label{functional}
\inf_{P\in\mathcal{A}} \E\left[\int_0^\infty e^{-\rho t} \left(C(X_t^P)\d t+K^+ \d P_t^+ +K^-\d P_t^-\right)\right],
\end{equation}
where 
$\E$ denotes the expectation under the probability measure $\P$, 
$\rho>0$ is the decision maker's discount rate of future costs, 
$K^+,K^->0$ are the marginal costs per unit of control exerted on $X^P$, and 
$C:\R\to\R^+$ is a holding and shortage cost function which satisfies the following \textbf{standing assumption}.

\begin{assumption}
\label{ass:C} 
There exists constants $p\geq2$, $\alpha_0,\alpha_1,\alpha_2>0$ such that the following hold true:
\begin{enumerate} 
\item[(i)] $0  \leq C(x) \leq \alpha_0(1 + |x|^p),$ for every $x \in \R$;

\vspace{-1mm}
\item[(ii)] $
|C(x)-C(x')|\leq \alpha_1 \big(1 + C(x)+ C(x')\big)^{1-\frac{1}{p}} |x-x'|,
$ for every $x,x'\in\R$;
\item[(iii)]  
$0 \leq  \lambda C(x)+(1-\lambda)C(x')-C(\lambda x + (1-\lambda) x') \leq \alpha_2 \lambda(1-\lambda)(1 + C(x) + C(x'))^{\left(1-\frac{2}{p}\right)}|x-x'|^2$, for every $x,x'\in\R$ and $\lambda \in(0,1)$; 
\item[(iv)] $\lim_{x \to \pm \infty}C'(x)= \pm \infty.$
\end{enumerate}
\end{assumption}
Notice that Assumption \ref{ass:C}.$(iii)$ above implies that $C$ is convex and locally semiconcave. Hence, by \cite[Corollary 3.3.8]{CS}, we have $C\in C^{1,\text{Lip}}_{\text{loc}}(\R;\R^+)$ (the class of continuously differentiable functions, whose first derivative is locally Lipschitz), so that the derivative in $(iv)$ exists. A classical quadratic cost $C(x) = (x-\overline{x})^2$, for some target level $\overline{x} \in \R$, clearly satisfies Assumption \ref{ass:C}.

Given the feature of a non-observable $\mu$, \eqref{functional} is not Markovian and cannot be therefore tackled via a dynamic programming approach. 
We derive below a new equivalent Markovian problem under full information, the so-called ``separated problem''. This will be then solved by exploiting its connection to a zero-sum game of optimal stopping and by a careful analysis of the regularity of its value function.

\subsection{The separated problem}
\label{sec:sepprobl}

In order to derive the equivalent problem under full information, we use standard arguments from filtering theory (see, e.g.~\cite[Section 4.2]{LSbook}) and we define the ``belief'' process 
$\Pi_t:=\mathbb{P}(\mu=\mu_1\,|\,\mathcal{F}^S_t)$, $t\geq0,$
according to which, decision makers update their beliefs on the (true) value of the drift $\mu$ based on the arrival of new information via the observation of the process $S$. Then, the dynamics of $X^P$ and $\Pi$ can be written as
\begin{equation}
\label{XPi}
\begin{cases}
\d X^P_t = (\mu_1\Pi_t + \mu_0(1-\Pi_t))\d t + \eta \d W_t+\d P_t, 
& X^P_{0^-}=x\in \R,\\
\d \Pi_t \;\,=\gamma \Pi_t(1-\Pi_t)\d W_t, 
& \Pi_0=\pi \in (0,1),
\end{cases}
\end{equation}
where the innovation process $W$ is an $\F^S$-Brownian motion on $(\Omega,\mathcal{F},\P)$ according to L\'evy's characterisation theorem (see, e.g., \cite[Theorem 4.1]{LSbook}), and 
$\gamma:=({\mu_1-\mu_0})/{\eta}>0.$
The triplet $(X^P, \Pi, P)$ is an $\F^S$-adapted time-homogeneous process on $(\Omega,\mathcal{F},\P)$. 
In \eqref{XPi}, the (unknown/non-observable) drift $\mu$ of $X$ in the original model is replaced with its filtering estimate $\E[\mu \,|\, \mathcal{F}^S_t]$. Moreover, the belief (learning) process $\Pi = (\Pi_t)_{t\geq 0}$ involved in the filtering is a bounded martingale on $[0,1]$ such that $\Pi_\infty \in \{0,1\}$, due to the fact that all information eventually gets revealed at time $t=\infty$.

Then, for $(X^P,\Pi)$ as in \eqref{XPi}, with $(x,\pi) \in \mathcal{O}:=\R \times (0,1)$, we define 
\begin{equation}
\label{functional2}
V(x,\pi):=\inf_{P\in\mathcal{A}} \E\left[\int_0^\infty e^{-\rho t} \left(C(X^P_t)\d t+K^+ \d P_t^+ +K^-\d P_t^-\right)\right],
\end{equation}
where all processes involved are now $\F^S$-adapted. Hence, \eqref{functional2} is a two-dimensional Markovian singular stochastic control problem with controls of bounded variation. 
By uniqueness of the strong solution to the belief equation, a control $P^{\star}$ is optimal for \eqref{functional} if and only if it is optimal for \eqref{functional2}, and the values in \eqref{functional} and \eqref{functional2} coincide.

Note that, in light of the dynamics of $(X^P,\Pi)$ in \eqref{XPi}, a high value of $\Pi$ close to $1$ would imply that the decision maker has a strong belief in a high drift $\mu_1$, while a low $\Pi$ close to $0$ would imply, on the contrary, a strong belief in a low drift $\mu_0$ scenario.

\begin{remark}[Full information cases]
In the formulation \eqref{functional}, the case of prior belief $\pi:=\P(\mu=\mu_1) \in \{0,1\}$ implies the certainty of the decision maker regarding whether $\mu=\mu_0$ or $\mu=\mu_1$. Hence, in this case, there is no uncertainty about the value of the drift $\mu$, which is not a random variable any more. 
Respectively, in the formulation \eqref{functional2}, the case of prior belief $\Pi_0=\pi \in \{0,1\}$ yields that the belief process $\Pi$ will actually remain constant through time, due to its dynamics which imply that $\Pi_t=\pi$ for all $t>0$.
Therefore, we equivalently have that such values of $\pi \in \{0,1\}$ correspond to the full information cases. 

In these cases, the optimal control problem becomes a standard one-dimensional bounded-variation stochastic control problem, for which an early study can be found in \cite{HarrisonTak}. The resulting optimal control strategy is triggered by two constant boundaries within which the process $X^P$ is kept (via a Skorokhod reflection).
\end{remark}

Given the convexity of $C$ as in Assumption \ref{ass:C}, and the linear structure of $P \mapsto X^P$ in \eqref{XPi}, we can show the next result by following standard arguments based on Koml\'os' theorem (see, e.g., \cite[Proposition 3.4]{Federico2014}  or \cite[Theorem 3.3]{KW}).

\begin{proposition}
\label{prop:existence}
There exists an optimal control $P^\star$ for \eqref{functional2}. Moreover, this is unique (up to indistinguishability) if $C$ is strictly convex. 
\end{proposition}

\section{The First Related Optimal Stopping Game}
\label{sec:OSgame}

We now derive a zero-sum optimal stopping game (Dynkin game) related to $V$, and we provide preliminary properties of its value function and of the geometry of its state space. In this section, the uncontrolled process $X^0$ with $P_t \equiv 0$ for all $t \geq 0$ becomes involved in the analysis, so we recall from \eqref{XPi} that $(X^0_t,\Pi_t)_{t\geq0} \equiv(S_t, \Pi_t)_{t\geq0}$ is the two-dimensional strong Markov process solving 
\begin{equation}
\label{X0Pi}
\begin{cases}
\d X^0_t \,= (\mu_1\Pi_t + \mu_0(1-\Pi_t))\d t + \eta \d W_t, 
& X^0_{0}=x\in \R,\\
\d \Pi_t \;\,=\gamma \Pi_t(1-\Pi_t)\d W_t, 
& \Pi_0=\pi \in (0,1),
\end{cases}
\end{equation}

%Thanks to Proposition \ref{prop:existence} and Assumption \ref{ass:C}, the proof of the next result employs the connection between bounded variation control problems and Dynkin games \cite[Theorems 3.1, 3.2]{KW}, adapted to our infinite-time horizon discounted setting; for full details refer to \cite{FFRArxiv}.
\begin{proposition}
\label{contv}
Consider the process $(X^0_t,\Pi_t)_{t\geq0}$ defined in \eqref{X0Pi} and define
\begin{equation}
\label{optstop}
v(x,\pi):=\inf_{\sigma}\sup_{\tau}\E_{(x,\pi)}\bigg[\int_0^{\tau\wedge\sigma} \hspace{-3mm}e^{-\rho t}C'(X_t^0)\d t-K^+e^{-\rho \tau} \mathds{1}_{\{\tau<\sigma\}}+K^-e^{-\rho \sigma} \mathds{1}_{\{\tau>\sigma\}}\bigg]
\end{equation}
where the optimisation is taken over the set of $\F^{W}$-stopping times and $\E_{(x,\pi)}$ denotes the expectation conditioned on $(X^0_0,\Pi_0)=(x,\pi)\in \mathcal{O}$.
Consider also the control value function $V(x,\pi)$ defined in \eqref{functional2}. Then, we have the following properties: 
\begin{enumerate}
\item[(i)] $x \mapsto V(x,\pi)$ is differentiable and $v(x,\pi) = V_x(x,\pi)$.
\item[(ii)] $x \mapsto V(x,\pi)$ is convex and therefore $x \mapsto v(x,\pi)$ is nondecreasing.
\item[(iii)] $\pi \mapsto v(x,\pi)$ is nondecreasing.
\item[(iv
)] $(x,\pi) \mapsto v(x,\pi)$ is continuous on $\R\times (0,1)$.
\end{enumerate}
\end{proposition}

\begin{proof}
In this proof, whenever we need to stress the dependence of the state process on its starting point, we denote by $(X^{0;(x',\pi')}, \Pi^{\pi'})$ the unique strong solution to \eqref{X0Pi} starting at $(x',\pi')\in \mathcal{O}$ at time zero. 
We prove separately the four parts.

{\it Proof of (i).} Thanks to Proposition \ref{prop:existence}, it suffices to apply \cite[Theorem 3.2]{KW} upon setting $G\equiv 0$, $\gamma_t:=e^{-\rho t} K^+$, and $\nu_t:=e^{-\rho t} K^-$, for $t \geq 0$, we get
$$H(\omega,t,x):= e^{-\rho t} 
C\Big(x + \eta W_t(\omega) + \int_0^t \big(\mu_0 + (\mu_1 - \mu_0) \Pi_s(\omega) \big) \d s\Big), \, (\omega,t,x) \in \Omega \times \R_+ \times \R,$$
and noticing that the proof in \cite{KW} can be easily adapted to our infinite-time horizon discounted setting with right-continuous controls (see also \cite[Lemma A.1, Proposition 3.4]{Federico2014} for a proof in a related setting).

{\it Proof of (ii).} Denote by $(X^{P;(x,\pi)}, \Pi^{\pi})$ the unique strong solution to \eqref{XPi} when $(X^{P}_{0^-},\Pi_0)=(x,\pi)$. The convexity of $V(x,\pi)$ with respect to $x$, can be easily shown by exploiting the convexity of $C(x)$ and the linear structure of $(x,P)\mapsto X^{P;(x,\pi)}$, for any $P\in \mathcal{A}$ and $(x,\pi) \in \mathcal{O}$. The nondecreasing property of $v(\cdot,\pi)$ then follows from the fact that $v = V_x$ from part $(i)$.

{\it Proof of (iii).} Notice that 
$X^0_t= x + \eta W_t + \int_0^t \big(\mu_1 \Pi_s + \mu_0 (1 - \Pi_s)\big) \d s$, $t \geq 0$,
and that $\pi \mapsto \Pi^{\pi}$ is nondecreasing due to standard comparison theorems for strong solutions to one-dimensional stochastic differential equations \cite[Chapter 5.2]{KS}. Then, the claim follows from \eqref{optstop} and Assumption \ref{ass:C} according to which $x \mapsto C'(x)$ is nondecreasing.

{\it Proof of (iv).} By \cite[Theorem 3.1]{KW} and Proposition \ref{prop:existence} we know that, for any $(x,\pi) \in \mathcal{O}$, \eqref{optstop} admits a saddle point. 
Take $(x_n,\pi_n)\to (x,\pi)$ as $n\uparrow \infty$, and let $(\tau^{\star},\sigma^{\star})$ and $(\tau_n^{\star},\sigma_n^{\star})$ realise the saddle-points for $(x,\pi)$ and $(x_n,\pi_n)$, respectively.
Then, we have
\begin{align}
\label{eq:diffvs}
v(x,\pi)-v(x_n,\pi_n) &\leq \E\bigg[\int_0^{\tau^{\star}\wedge \sigma^{\star}_n} e^{-\rho t} \left(C'(X_t^{0;(x,\pi)})-C'(X_t^{0;(x_n,\pi_n)})\right)\d t\bigg]\nonumber \\
&\leq \E\bigg[\int_0^{\infty} e^{-\rho t} \left|C'(X_t^{0;(x,\pi)})-C'(X_t^{0;(x_n,\pi_n)})\right|\d t\bigg].
\end{align}
Without loss of generality, we can take $(x_n,\pi_n) \subset (x-\varepsilon, x + \varepsilon) \times (\pi-\varepsilon, \pi + \varepsilon)$, for a suitable $\varepsilon>0$ and for $n$ sufficiently large. Then, by Assumption \ref{ass:C}.$(ii)$ and standard estimates using Assumption \ref{ass:C}.$(i)$, the expression of $X^0$ and the fact that $\Pi$ is bounded in $[0,1]$, we can invoke the dominated convergence theorem and obtain
$\limsup_{n\to\infty} (v(x,\pi)-v(x_n,\pi_n))\leq 0.$
In order to evaluate the difference $v(x_n,\pi_n)-v(x,\pi)$, we now employ the couple of stopping times $(\tau_n^{\star},\sigma^{\star})$ and employ the same rationale leading to \eqref{eq:diffvs} so to obtain
$\limsup_{n\to\infty} (v(x_n,\pi_n)-v(x,\pi))\leq 0.$
Combining the last two inequalities, we obtain the desired continuity claim. 
\end{proof}
 
In the rest of this section, we focus on the study of the optimal stopping game $v$ presented in \eqref{optstop}, due to its connection to our stochastic control problem (cf.\ Proposition \ref{contv}). To that end, we define the so-called continuation (waiting) region 
\begin{equation}
\label{C1}
\mathcal{C}_1:=\big\{(x,\pi)\in\mathcal{O}: \ -K^+< {v}(x,\pi)<K^-\big\},
\end{equation}
and the stopping region ${\mathcal{S}_1} := {\mathcal{S}_1}^+ \cup {\mathcal{S}_1}^-$, whose components are given by 
\begin{equation}
\label{S1}
{\mathcal{S}_1}^+:=\big\{(x,\pi)\in\mathcal{O}: \ v(x,\pi)\leq - K^+\big\}, \ \ \ 
{\mathcal{S}_1}^-:=\big\{(x,\pi)\in\mathcal{O}: \ v(x,\pi)\geq K^-\big\}.
\end{equation}

In light of the continuity of ${v}$ in Proposition \ref{contv}.$(iv)$, we conclude that the continuation region $\mathcal{C}_1$ is an open set, while the two components of the stopping regions ${\mathcal{S}_1}^{\pm}$ are both closed sets.
We can therefore define the free boundaries
\begin{align} 
\label{a+-}
a_+(\pi) &:=\sup\big\{x\in\R: v(x,\pi)\leq -K^+\big\},  
a_-(\pi) :=\inf\big\{x\in\R: v(x,\pi)\geq K^-\big\}.
\end{align}
Here, and throughout the rest of this paper, we use the convention $\sup\emptyset = - \infty$ and $\inf \emptyset = + \infty$.
Then, by using the fact that $v$ is nondecreasing with respect to $x$ (see Proposition \ref{contv}.$(ii)$), we can obtain the structure of the continuation and stopping regions, which take the form
\begin{align}
\label{C1a}
&\mathcal{C}_1 =
\big\{(x,\pi)\in\mathcal{O}: \ a_+(\pi)<x<a_-(\pi)\big\}, \\
\label{S1a}
\mathcal{S}^+_1=\big\{(x,\pi)\in\;&\mathcal{O}: \ x \leq a_+(\pi)\big\} 
\quad \text{and} \quad    
\mathcal{S}^-_1=\big\{(x,\pi)\in\mathcal{O}: \ x \geq a_-(\pi)\big\}.
\end{align}
Clearly, the continuity of $v$ further implies that the free boundaries $a_\pm$ are strictly separated, namely
$a_+(\pi) < a_-(\pi) \quad \text{for all } \pi \in (0,1).$ 

We now prove some preliminary properties of the free boundaries $\pi \mapsto a_\pm(\pi)$. 

\begin{proposition}
\label{prop:aa}
The free boundaries $a_\pm$ defined in \eqref{a+-} satisfy: 
\begin{enumerate}
\item[(i)] $a_\pm(\cdot)$ are nonincreasing on $(0,1)$.
\item[(ii)] $a_+(\cdot)$ is left-continuous and $a_-(\cdot)$ is right-continuous on $(0,1)$.
\item[(iii)] There exist constants $x_\pm^* \in \R$, such that 
$x_+^*\leq a_+(\pi) < a_-(\pi)\leq x_-^*$, for all $\pi\in(0,1)$.
Moreover, letting $(C')^{-1}$ be the generalised inverse of $C'$, we have $a_+(\pi) \leq (C')^{-1}(-\rho K^+)$ and $a_-(\pi) \geq (C')^{-1}(\rho K^-)$ for all $\pi \in (0,1)$. 
\end{enumerate}
\end{proposition}
\begin{proof}
{\it Proof of (i).} 
This is a consequence of the definitions of $a_\pm(\cdot)$ in \eqref{a+-} and the fact that $v(x,\cdot)$ is nondecreasing for any $x \in \R$; cf.\ Proposition \ref{contv}.$(iii)$.

{\it Proof of (ii).} 
This follows from part $(i)$ above and the closedness of the sets ${\mathcal{S}_1}^{\pm}$.

{\it Proof of (iii).} 
The fact that $a_+(\pi) \leq (C')^{-1}(-\rho K^+)$ and $a_-(\pi) \geq (C')^{-1}(\rho K^-)$ follows by noticing that $\mathcal{S}^+_1 \subseteq \{x\in \R:\, x \leq (C')^{-1}(-\rho K^+)\}$ and $\mathcal{S}^-_1 \subseteq \{x\in \R:\, x \geq (C')^{-1}(\rho K^-)\}$. 
This can be seen by observing that an integration by parts yields
\begin{align*}
v(x,\pi)= \inf_{\sigma}\sup_{\tau}\E_{(x,\pi)}\bigg[&\mathds{1}_{\{\tau<\sigma\}} \bigg(\int_0^{\tau} e^{-\rho t}\Big(C'(X_t^0) + \rho K^+\Big)\d t-K^+\bigg) \nonumber \\
&+ \mathds{1}_{\{\tau>\sigma\}}\bigg(\int_0^{\sigma} e^{-\rho t}\Big(C'(X_t^0) - \rho K^-\Big)\d t + K^-\bigg)\bigg];
\end{align*}
hence, independently of the choice of the other player, the sup-player will never stop in the region $\{(x,\pi) \in \mathcal{O}:\, C'(x) + \rho K^+ >0\}$ and the inf-player does not stop in the region $\{(x,\pi) \in \mathcal{O}:\, C'(x) - \rho K^- < 0\}$.

In order to show the other bounds, we proceed as follows. Since $\mu_1>\mu_0$ and $\Pi_t \in(0,1)$, we have $\P_{(x,\pi)}$-a.s., for any $t\geq0$, that 
$X_t^0 \geq x+\eta W_t+\mu_0 t=:\underline{X}_t^0$
and 
$X_t^0 \leq x+\eta W_t+\mu_1 t=:\overline{X}_t^0.$
Therefore, the latter two estimates yield that $\underline{X}_t^0 \leq X_t^0 \leq \overline{X}_t^0$ for all $t \geq 0$. 
Combining these inequalities with the fact that $C'(\cdot)$ is nondecreasing due to Assumption \ref{ass:C} and the definition \eqref{optstop} of the value function $v(x,\pi)$, we conclude that
\begin{equation}
\label{boundsv}
v_0(x) \leq v(x,\pi) \leq v_1(x), \quad \text{for all} \quad (x,\pi)\in\mathcal{O},
\end{equation}
where we have introduced the one-dimensional optimal stopping games 
\begin{align*}
v_0(x)&:=\inf_{\sigma\in\mathcal{T}}\sup_{\tau\in\mathcal{T}} \E\bigg[\int_0^{\tau\wedge\sigma} e^{-\rho t} C'(\underline{X}_t^0)\d t -K^+e^{-\rho \tau} \mathbf{1}_{\{\tau<\sigma\}}+K^-e^{-\rho \sigma} \mathbf{1}_{\{\tau>\sigma\}} \bigg] \\
v_1(x)&:=\inf_{\sigma\in\mathcal{T}}\sup_{\tau\in\mathcal{T}} \E\bigg[\int_0^{\tau\wedge\sigma} e^{-\rho t} C'(\overline{X}_t^0)\d t -K^+e^{-\rho \tau} \mathbf{1}_{\{\tau<\sigma\}}+K^-e^{-\rho \sigma} \mathbf{1}_{\{\tau>\sigma\}} \bigg].
\end{align*}
Because both $v_0(\cdot)$ and $v_1(\cdot)$ are nondecreasing on $\R$, standard techniques allow to show that due to Assumption \ref{ass:C}.$(iv)$ there exists finite $x_-^{\star}, x_+^{\star}$ such that
$\{x\in\R: \ x\geq x_-^{\star}\}=\{x\in\R: \ v_0(x) \geq K^-\}$ and 
$\{x\in\R: \ x\leq x_+^{\star}\}=\{x\in\R: \ v_1(x)\leq -K^+\}$.
Hence, combining the latter two regions together with the inequalities in \eqref{boundsv}, we eventually get that
\begin{align} \label{s-}
\begin{split}
&\{x\in\R: \ x\geq x_-^{\star}\} \subseteq\{(x,\pi)\in\mathcal{O}: \ v(x,\pi)\geq K^-\}=\mathcal{S}_1^- , \\
%\label{s+}
&\{x\in\R: \ x\leq x_+^{\star}\} \subseteq\{(x,\pi)\in\mathcal{O}: \ v(x,\pi)\leq -K^+\}=\mathcal{S}_1^+. 
\end{split}
\end{align}
Hence, $\mathcal{S}_1^\pm\neq \emptyset$ and the claim follows from \eqref{s-}. %-\eqref{s+}.
\end{proof}

\section{A Decoupling Change of Measure}
\label{sec:2transf}

In order to provide further results about the optimal control problem \eqref{functional2} and the associated Dynkin game \eqref{optstop}, it is convenient to decouple the dynamics of the controlled inventory process  $X^P$ and the belief process $\Pi$. This can be achieved via a transformation of state space and a change of measure, as we explain in the following subsections.

\subsection{Transformation of process $\Pi$ to $\Phi$} \label{TransPhi}
We first recall from \eqref{XPi} (see also \eqref{X0Pi}), that for any prior belief $\Pi_0 = \pi \in (0,1)$, we have $\Pi_t\in(0,1)$ for all $t\in(0,\infty)$. Hence, we define the process 
$\Phi_t:={\Pi_t}/({1-\Pi_t})$, $t\geq 0,$
whose dynamics are given via It\^o's formula by  
\begin{equation}
\label{dynPhi0}
\d \Phi_t %= \gamma^2\Pi_t\Phi_t\d t + \gamma \Phi_t\d W_t
=\gamma\Phi_t (\gamma \Pi_t\d t+\d W_t), 
\quad \Phi_0=\varphi:=\tfrac{\pi}{1-\pi}.
\end{equation}
Note that, the process $\Phi$ is known as the ``likelihood ratio process'' in the literature of filtering theory (see, e.g. \cite{JP17}).

\subsection{Change of measure from $\P$ to $\Q_T$, for some fixed $T>0$} \label{QT}
We begin by defining the exponential martingale
$\zeta_T:= \exp \{-\gamma\int_0^T\Pi_s\d W_s -\frac{1}{2}\int_0^T\gamma^2\Pi_s^2\d s \},$
and the measure ${\Q}_T\sim \P$ on $(\Omega,\mathcal{F}_T)$ by
${\d\Q_T}/{\d\P}=\zeta_T.$

Then, the process 
${W}^*_t:=W_t+\gamma \int_0^t \Pi_s\d s$, $t\in[0,T]$, 
is a Brownian motion in $[0,T]$ under $\Q_T$, and the dynamics of $\Phi$ in \eqref{dynPhi0} simplifies to 
%\begin{equation} \label{dynPhi}
%\d \Phi_t=\gamma\Phi_t\d{W}^*_t, \quad t\in(0,T], \quad \Phi_0=\varphi ,
%\end{equation}
$\d \Phi_t=\gamma\Phi_t\d{W}^*_t$, $t\in(0,T]$, $\Phi_0=\varphi$,
hence $\Phi$ is an exponential martingale under $\Q_T$.
Consequently, applying the same change of measure to the process $X^P$ from \eqref{XPi}, 
%\begin{equation}
%\label{dynXP}
%\d X^P_t=\mu_0\d t + \eta \d {W}^*_t+\d P_t^+-\d P_t^-, \quad t\in[0,T], \quad X^P_{0-}=x.
%\end{equation}
we obtain 
$\d X^P_t=\mu_0\d t + \eta \d {W}^*_t+\d P_t^+-\d P_t^-$, $t\in[0,T]$, $X^P_{0-}=x$.

In order to change the measure also in the cost criterion of our value function in \eqref{functional2}, we further define the process
$Z_t:=({1+\Phi_t})/({1+\varphi})$, $t\in[0,T],$
which can be verified via It\^o's formula to satisfy 
$Z_t= 1/\zeta_t$, {for every } $t\in[0,T]$. 
Hence, denoting by $\E^{\Q_T}$ the expectation under $\Q_T$, we have that
\begin{align} \label{E=EQ}
&\E\bigg[\int_0^T e^{-\rho t} \left(C(X^P_t)\d t+ K^+\d P_t^+ + K^- \d P_t^-\right)\bigg] \notag\\&
=\frac{1}{1+\varphi}\,\E^{\Q_T}\bigg[(1+\Phi_T)\int_0^T e^{-\rho t} \Big(C(X^P_t)\d t+ K^+\d P_t^+ + K^- \d P_t^-\Big)\bigg].
\end{align}
Since the process $(1 + \Phi_t)_{t\geq0}$ defines a nonnegative martingale under $\Q_T$, by an application of It\^o's formula we can write
\begin{align*}
&\E^{\Q_T}\bigg[ (1+\Phi_T)\int_0^Te^{-\rho t} C(X^P_t)\d t\bigg]=\E^{\Q_T}\bigg[\int_0^T e^{-\rho t}(1+\Phi_t)C(X^P_t)\d t\bigg], \\
&\E^{\Q_T}\bigg[(1+\Phi_T)\int_0^T e^{-\rho t}\d P_t^{\pm} \bigg]=\E^{\Q_T}\bigg[\int_0^T e^{-\rho t} (1+\Phi_t)\d P^{\pm}_t\bigg].
\end{align*}
Hence, combining together the above expressions of the expectations $\E^{\Q_T}$ we get that \eqref{E=EQ} can be expressed in the form of 
\begin{align}
\label{eq:transfJ}
& \E\bigg[\int_0^T e^{-\rho t} \Big(C(X^P_t)\d t+ K^+\d P_t^+ + K^- \d P_t^-\Big)\bigg] \nonumber \\& 
=\frac{1}{1+\varphi}\,\E^{\Q_T}\bigg[\int_0^T e^{-\rho t}(1+\Phi_t)\Big(C(X^P_t)\d t+ K^+\d P_t^+ +K^- \d P_t^-\Big)\bigg].
\end{align}

\subsection{Passing to the limit as $T\to\infty$ and to the new measure $\Q$} \label{Q}
We firstly notice that passing to the limit as $T\to\infty$ cannot be performed directly to the latter expression in \eqref{eq:transfJ}, since the measure $\Q_T$ changes with $T$. 
{Nevertheless, noticing that the right-hand side of \eqref{eq:transfJ} only depends on the law of the processes involved we can introduce a new auxiliary problem. 

To that end, first of all note that any $P\in\mathcal{A}$ has paths that are right-continuous and (locally) of bounded variation $\Q_T$-a.s. and it is $\mathbb{F}^S$-adapted since $\mathbb{F}^S = \mathbb{F}^W = \mathbb{F}^{W^*}$.
Then, define a new complete probability space $(\overline{\Omega},\overline{\mathcal{F}}, \overline{\Q})$ supporting a Brownian motion $(\overline{W}_t)_{t\geq 0}$, let $(\overline{\mathcal{F}}^o_t)_{t\geq0}$ be the raw filtration generated by $\overline{W}$, and denote by $\overline{\mathbb{F}}:=(\overline{\mathcal{F}}_t)_{t\geq0}$ its augmentation with the $\overline{\Q}$-null sets.
Hence, introducing
\begin{eqnarray*}
\overline{\mathcal{A}}&:=&\big\{\overline{P}:\overline{\Omega}\times\R^+\to\R \ \mbox{such that} \ t\mapsto \overline{P}_t \ \mbox{is right-continuous, (locally) of bounded } \\ \nonumber 
&& \mbox{variation and}\,\, \overline{P} \ \mbox{is}  \ \overline{\F}-\mbox{adapted}\big\},
\end{eqnarray*}
by \cite[Lemma 5.5]{DeAMi2021} (adjusted to our setting with right-continuous controls), given $P\in\mathcal{A}$ there exists $\overline{P} \in \overline{\mathcal{A}}$ that is $\overline{\mathcal{F}}^o_{t+}-$predictable and such that $\text{Law}_{\Q_T}({W}^*,P) = \text{Law}_{\overline{\Q}}(\overline{W},\overline{P}).$ This in turn leads to (cf.\ \cite[Corollary 5.6]{DeAMi2021}) 
\begin{equation}
\label{eq:eqivLaw}
\text{Law}_{\Q_T}({W}^*, X^P, \Phi, P) = \text{Law}_{\overline{\Q}}(\overline{W},\overline{X}^{\overline{P}}, \overline{\Phi}, \overline{P}),
\end{equation}
where $(\overline{X}^{\overline{P}}, \overline{\Phi})$ is the strong solution on $(\overline{\Omega},\overline{\mathcal{F}}, \overline{\mathbb{F}}, \overline{\Q})$ to the controlled stochastic differential equation 
$$
\begin{cases}
\d \overline{X}^{\overline{P}}_t = \mu_0\d t + \eta \d \overline{W}_t+\d \overline{P}_t^+-\d \overline{P}_t^-, \quad &\overline{X}^{\overline{P}}_{0-}=x,\\
\d \overline{\Phi}_t \;\,= \gamma\overline{\Phi}_t\d\overline{W}_t, \quad &\overline{\Phi}_0=\varphi:=\frac{\pi}{1-\pi},
\end{cases}
$$ 
with $\overline{P}^{\pm}$ denoting the nondecreasing processes providing the minimal decomposition of $\overline{P} \in \overline{\mathcal{A}}$ as $\overline{P}=\overline{P}^+-\overline{P}^-$.

Denoting now by $\overline{\E}$ the expectation on $(\overline{\Omega},\overline{\mathcal{F}})$ under $\overline{\Q}$, we have for every $T>0$, 
\begin{align*}
&\E^{\Q_T}\bigg[\int_0^T e^{-\rho t}(1+\Phi_t)\left(C(X^P_t)\d t+ K^+\d P_t^++K^- \d P_t^-\right)\bigg] \\&
= \overline{\E}\bigg[\int_0^T e^{-\rho t}(1+\overline{\Phi}_t)\Big(C(\overline{X}^{\overline{P}}_t)\d t+ K^+\d \overline{P}_t^+ + K^- \d \overline{P}_t^-\Big)\bigg], 
\end{align*}
due to \eqref{eq:eqivLaw}.
Therefore, combining the above equality with \eqref{eq:transfJ}, we eventually get 
\begin{align} \label{E=phiE}
&\E\bigg[\int_0^T e^{-\rho t} \Big(C(X^P_t)\d t+ K^+\d P_t^+ +K^- \d P_t^-\Big)\bigg] \nonumber \\& 
=\frac{1}{1+\varphi}\,\overline{\E}\bigg[\int_0^T e^{-\rho t}(1+\overline{\Phi}_t)\Big(C(\overline{X}^{\overline{P}}_t)\d t+ K^+\d \overline{P}_t^+ + K^- \d \overline{P}_t^-\Big)\bigg],
\end{align}

Thanks to \eqref{E=phiE}, %the latter equation, 
we can now take limits as $T\to\infty$ and obtain, in view of the definitions \eqref{functional2} of the control value function and \eqref{dynPhi0} of the starting value $\varphi$, that
\begin{equation} \label{VVbar}
\begin{split}
&V(x,\pi) = (1-\pi)\overline{V}\big(x,\tfrac{\pi}{1-\pi}\big), 
\;\; \text{or equivalently} \;\; \overline{V}(x,\varphi)=(1+\varphi)V\big(x, \tfrac{\varphi}{1+\varphi}\big), 
\\
%\end{equation}
%where we define
%\begin{equation*}
&\text{where} \;\;
\overline{V}\left(x,\varphi\right):=\inf_{\overline{P}\in\overline{\mathcal{A}}} \overline{\E}\bigg[\int_0^\infty \hspace{-3mm}e^{-\rho t}(1+\overline{\Phi}_t)\Big(C(\overline{X}^{\overline{P}}_t)\d t + K^+\d \overline{P}_t^+ + K^-\d \overline{P}_t^-\Big)\bigg].
\end{split}
\end{equation} 
Therefore, in order to obtain the value function $V(x,\pi)$ from  \eqref{functional2}, we could instead solve first the above problem to get $\overline{V}\left(x,\varphi\right)$ and then use the equality in \eqref{VVbar}.  
However, in order to simplify the notation, {\textbf{from now on}} in the study of $\overline{V}$ we will simply write $(\Omega, \mathcal{F}, \mathbb{F}, {\Q}, {\E}^{\Q}, {W}, {X}, {\Phi}, {P}, {\mathcal{A}})$ instead of $(\overline{\Omega}, \overline{\mathcal{F}}, \overline{\mathbb{F}}, \overline{\Q}, \overline{\E}, \overline{W}, \overline{X}, \overline{\Phi}, \overline{P}, \overline{\mathcal{A}})$.

\subsection{The optimal control problem with state-space process $(X^P,\Phi)$ under the new measure $\Q$} \label{VXPhi}
Summarising the results from Sections \ref{TransPhi}--\ref{Q}, we henceforth focus on the study of the following optimal control problem 
\begin{align}
\label{eq:Vbar}
\begin{split}
\overline{V}\left(x,\varphi\right) 
&:= \hspace{-1mm}\inf_{{P}\in {\mathcal{A}}} {\E}^{\Q}\bigg[\int_0^\infty \hspace{-3mm}e^{-\rho t}(1+{\Phi}_t)\Big(C({X}^{{P}}_t)\d t + K^+\d {P}_t^+ + K^-\d {P}_t^-\Big)\bigg] \\ 
&\hspace{-1mm}=: \inf_{{P}\in {\mathcal{A}}}\overline{\mathcal{J}}_{x,\varphi}(P).
\end{split}
\end{align}
under the dynamics
\begin{equation}
\label{eq:XPhiP}
\begin{cases}
\d X^P_t = \mu_0\d t + \eta \d W_t+\d P_t^+-\d P_t^-, \quad &X^{P}_{0-}=x  \in \R,\\
\d \Phi_t \;\,= \gamma\Phi_t\d W_t, \quad &\Phi_0=\varphi:=\frac{\pi}{1-\pi} \in (0,\infty),
\end{cases}
\end{equation}
for a standard Brownian motion $W$. In light of the equality in \eqref{VVbar}, this will lead to the original value function $V(x,\pi)$ from  \eqref{functional2}. 
In the remaining of Section \ref{sec:2transf}, we expand our study -- beyond the values of the control problems -- to the relationship between the free boundaries in the two formulations, since these boundaries will eventually define the optimal control strategy (see Section \ref{sec:Verif}).

\subsection{The optimal stopping game associated to \eqref{eq:Vbar}--$\eqref{eq:XPhiP}$ under the new measure $\Q$} \label{vXPhi}
\label{barVprop}

The next result is concerned with properties of the value function defined in \eqref{eq:Vbar} and its connection to an associated optimal stopping game. The first existence claim follows from Proposition \ref{prop:existence}, since existence of an optimal control is preserved under the change of measure performed in the previous section. The second claim can be proved by employing arguments similar to those used in the proof of Proposition \ref{contv} above. Hence, the proof is omitted for brevity.

\begin{proposition}
\label{prop:Vbar}
Consider the problem defined in \eqref{eq:Vbar}--$\eqref{eq:XPhiP}$.  
\begin{enumerate}
\item[(i)] There exists an optimal control $P^{\star}$ solving \eqref{eq:Vbar}. Moreover, $P^{\star}$ is unique (up to indistinguishability) if $C$ is strictly convex.
\item[(ii)] $x \mapsto \overline{V}(x,\varphi)$ is convex and differentiable, such that $\overline{V}_x(x,\varphi)=\overline{v}(x,\varphi)$ on $\R\times (0,\infty)$, for 
\begin{align}
\label{eq:barv}
\vspace{-3mm}\bar{v}(x,\varphi) :=\inf_{\sigma}\sup_{\tau} \E^{\Q}\bigg[ &\int_0^{\tau\wedge\sigma} e^{-\rho t}(1+\Phi_t)C'(X_t^0)\d t 
- K^+(1+\Phi_\tau)e^{-\rho \tau} \mathbf{1}_{\{\tau<\sigma\}} \notag\\
&+K^-(1+\Phi_\sigma)e^{-\rho \sigma} \mathbf{1}_{\{\tau>\sigma\}}\bigg],
\end{align}
over the set of $\F^W$-stopping times and state-space process given by 
\begin{equation}
\label{XoPhi}
\begin{cases}
\d {X}^0_t = \mu_0 \d t + \eta \d W_t, \quad &X^0_{0}=x \in \R,\\
\d \Phi_t \;\,= \gamma \Phi_t\d W_t, \quad &\Phi_0=\varphi:=\frac{\pi}{1-\pi} \in (0,\infty).
\end{cases}
\end{equation}
\end{enumerate}
\end{proposition} 

It further follows from the previous analysis, namely Sections \ref{TransPhi}--\ref{Q}, that the value function $v(x,\pi)$ of the optimal stopping game in \eqref{optstop} is connected to the value function $\bar{v}(x,\varphi)$ of the new game introduced above in \eqref{eq:barv}, according to (see also \eqref{VVbar} for the control value functions) the following equality 
\begin{equation}
\label{eq:v}
\bar{v}(x,\varphi)=(1+\varphi)\,v\big(x,\tfrac{\varphi}{1+\varphi}\big) .
\end{equation}
In view of the above relationship, the value function $\bar{v}(\cdot,\cdot)$ inherits important properties which have already been proved for $v(\cdot,\cdot)$ in Section \ref{sec:OSgame}. In particular, we have directly from Proposition \ref{contv}.$(ii)$ and $(iv)$ the following result.
\begin{proposition}
\label{prop:barv}
The value function $\bar{v}$ defined in \eqref{eq:barv} satisfies: 
\begin{enumerate}
\item[(i)] $(x,\varphi) \mapsto \bar{v}(x,\varphi)$ is continuous over $\R\times (0,\infty)$;
\item[(ii)] $x \mapsto \bar{v}(x,\varphi)$ is nondecreasing.
\end{enumerate}
\end{proposition}

Following similar steps as in Section \ref{sec:OSgame} to study the new game \eqref{eq:barv}, we define below the so-called continuation (waiting) region 
\begin{equation}
\label{C2}
 {\mathcal{C}_2}:=\big\{(x,\varphi)\in\R\times (0,\infty): \ -K^+(1+\varphi)< \bar{v}(x,\varphi)<K^-(1+\varphi)\big\},
\end{equation}
and the stopping region ${\mathcal{S}_2} := {\mathcal{S}_2}^+ \cup {\mathcal{S}_2}^-$, whose components are given by 
\begin{align}
\begin{split}
{\mathcal{S}_2^+}&:=\big\{(x,\varphi)\in\R\times (0,\infty): \ \bar{v}(x,\varphi)\leq - K^+(1+\varphi)\big\}, \\    
{\mathcal{S}_2^-}&:=\big\{(x,\varphi)\in\R\times (0,\infty): \ \bar{v}(x,\varphi)\geq K^-(1+\varphi)\big\}.
\end{split}
\label{S2}
\end{align}
Moreover, in light of the continuity of $\bar{v}$ in Proposition \ref{prop:barv}.$(i)$, we conclude that the continuation region $\mathcal{C}_2$ is an open set, while the two components of the stopping regions ${\mathcal{S}_2}^{\pm}$ are both closed sets.
We can therefore define the free boundaries
\begin{align}
\label{eq:b+}
\begin{split}
b_+(\varphi)&:=\sup\big\{x\in\R: \ \overline{v}(x,\varphi)\leq K^+(1+\varphi)\big\}, \\
%\label{eq:b-}
b_-(\varphi)&:=\inf\{x\in\R: \ \overline{v}(x,\varphi)\geq K^-(1+\varphi)\}.
\end{split}
\end{align}
Then, by using the fact that $\bar{v}$ is nondecreasing with respect to $x$ (see Proposition \ref{prop:barv}.$(ii)$), we can obtain the structure of the continuation and stopping regions, as 
\begin{align}
\label{C2b}
\begin{split}
&{\mathcal{C}_2}=\big\{(x,\varphi)\in\R\times (0,\infty): \ b_+(\varphi)<x<b_-(\varphi)\big\}, \\
%\label{S2b}
{\mathcal{S}_2^+} \hspace{-1mm}= \hspace{-1mm}\big\{(x,\varphi)\hspace{-0.5mm}\in \hspace{-0.5mm}\R &\times \hspace{-0.5mm}(0,\infty)\hspace{-0.5mm}: \hspace{-0.5mm}x \hspace{-0.5mm}\leq \hspace{-0.5mm}b_+(\varphi)\big\}, \;
{\mathcal{S}_2^-} \hspace{-1mm}= \hspace{-1mm}\big\{(x,\varphi)\hspace{-0.5mm}\in \hspace{-0.5mm}\R\times \hspace{-0.5mm}(0,\infty)\hspace{-0.5mm}: \hspace{-0.5mm}b_-(\varphi) \hspace{-0.5mm}\leq \hspace{-0.5mm}x \big\}.
\end{split}
\end{align}
Clearly, the continuity of $\bar{v}$ implies that these free boundaries $b_\pm$ are strictly separated, namely $b_+(\varphi) < b_-(\varphi)$ for all $\varphi \in (0,\infty)$. 

Moreover, observe that the relationship in \eqref{eq:v} together with the definitions \eqref{C1} and \eqref{C2} of $\mathcal{C}_1$ and $\mathcal{C}_2$, respectively, imply that the latter two regions are equal under the transformation from $(x,\pi)$- to $(x,\varphi)$-coordinates.
To be more precise, for any $(x,\pi) \in \R \times (0,1)$, define the transformation
$\overline T
:=(\overline T_1, \overline T_2):\R\times(0,1)\to\R\times(0,\infty)$, 
by $(\overline T_1(x,\pi), \overline T_2(x,\pi)) 
=(x,\frac \pi{1-\pi}),$
which is invertible and its inverse is given by
$\overline T^{-1}(x,\varphi) 
=(x, \frac{\varphi}{1+\varphi})$,
for $(x,\varphi)\in \R\times(0,\infty).$
Hence, $\overline T: \R\times(0,1)\to\R\times(0,\infty)$ is a global diffeomorphism, which implies together with the expressions of \eqref{C1}--\eqref{S1} and \eqref{C2}--\eqref{S2} that 
${\mathcal{C}_2}=\overline T({\mathcal{C}_1})$ and 
${\mathcal{S}_2^\pm}=\overline T({\mathcal{S}_1^{\pm}}).$
Taking this into account together with the expressions \eqref{C1a}--\eqref{S1a} of $\mathcal{C}_1$ and $\mathcal{S}_1^{\pm}$, we can further conclude from the expressions \eqref{C2b} %--\eqref{S2b} 
of $\mathcal{C}_2$ and $\mathcal{S}_2^{\pm}$ that 
\begin{equation} \label{b=a}
b_\pm(\varphi)=a_\pm\big(\tfrac{\varphi}{1+\varphi}\big).
\end{equation}

Hence, in light of the previously proved results for $a_\pm$ in Proposition \ref{prop:aa}, we also obtain the following preliminary properties of the free boundaries $\varphi \mapsto b_\pm(\varphi)$.
\begin{proposition} \label{prop:barb}
The free boundaries $b_\pm$ defined in \eqref{eq:b+} %--\eqref{eq:b-}  
satisfy: 
\begin{enumerate}
\item[(i)] $b_\pm(\cdot)$ are nonincreasing on $(0,\infty)$.
\item[(ii)] $b_+(\cdot)$ is left-continuous and $b_-(\cdot)$ is right-continuous on $(0,\infty)$.
\item[(iii)] $b_\pm(\cdot)$ are bounded by $x^*_{\pm}$ as in Proposition \ref{prop:aa}:
$x_+^*\leq b_+(\varphi)<b_-(\varphi)\leq x_-^*$, for all $\varphi\in(0,\infty)$.
Moreover, we have $b_+(\varphi) \leq (C')^{-1}(-\rho K^+)$ and $b_-(\varphi) \geq (C')^{-1}(\rho K^-)$ for all $\varphi \in (0,\infty)$.
\end{enumerate}
\end{proposition}
Notice that the explicit relationship \eqref{b=a} between the free boundaries $a_{\pm}$ and $b_{\pm}$ that we proved above, is not only crucial for retrieving the original boundaries $a_{\pm}$ from $b_{\pm}$, but it is also particularly useful in the proof of Proposition \ref{prop:barb}.$(i)$ and $(iii)$. In fact, proving the monotonicity and boundedness of $b_{\pm}$ by directly working on the Dynkin game \eqref{eq:barv} is not a straightforward task.

Up this point, we managed to obtain the structure of the optimal stopping strategies and preliminary properties of the corresponding optimal stopping boundaries associated with these strategies, for both Dynkin games \eqref{optstop} and \eqref{eq:barv} connected to the optimal control problems 
 \eqref{functional2} and \eqref{eq:Vbar}, respectively. Moreover, we managed to obtain some regularity results for the latter control value functions (see Propositions \ref{contv}, \ref{prop:Vbar} and \ref{prop:barv}).
In Sections \ref{sec:HJB} and \ref{sec:Verif} below, building on the aforementioned analysis, we show that the control value function $\overline{V}$ has the sufficient regularity needed to construct an optimal control strategy. This will involve the boundaries $b_\pm$.

\section{HJB Equation and Regularity of $\overline{V}$}
\label{sec:HJB}

In this section, we introduce the Hamilton-Jacobi-Bellman (HJB) equation (variational inequality) associated to the control value function $\overline{V}$ defined in \eqref{eq:Vbar} and state-space process $(X^P,\Phi)$ given by \eqref{eq:XPhiP}.
First, let $\mathcal{D}\subseteq \R^2$ be an open domain and define the space $C^{k,h}(\mathcal{D};\R)$ as the space of functions $f:\mathcal{D}\to\R$ which are $k$-times continuously differentiable with respect to the first variable and $h$-times continuously differentiable with respect to the second variable. When $k=h$ we simply write $C^h$. 

We begin our study with the following \emph{ex ante} regularity result for $\overline{V}$. 
Its technical proof can be found in 
% the extended version of this paper \cite{FFRArxiv}.
%%%
Appendix \ref{A}.
%%%
\begin{proposition} 
\label{prop:semiconc}
The control value function $\overline{V}$ defined in \eqref{eq:Vbar} is locally semiconcave; that is, for every $R>0$ there exists $L_{R}>0$ such that for all $\lambda\in[0,1]$ and all $(x,\varphi),(x',\varphi')$ such that $|(x,\varphi)|\leq R$ and $|(x',\varphi')|\leq R$, we have
$$
\lambda  \overline{V}(x,\varphi)+(1-\lambda) \overline{V}(x',\varphi')-\overline{V}(\lambda(x,\varphi)+(1-\lambda)(x',\varphi'))\leq L_{R}\lambda(1-\lambda)|(x,\varphi)-(x',\varphi')|^{2}.
$$
In particular, by \cite[Theorem 2.1.7]{CS}, we conclude that $\overline{V}$ is locally Lipschitz.
\end{proposition}

Given the locally Lipschitz continuity proved in the previous result, we now aim at employing the HJB equation to investigate further regularity of $\overline{V}$. 
To that end, we define on $f\in C^{2}(\R \times (0,\infty);\R)$ the second order differential operator
\begin{equation*}
%\label{Lf}
\mathcal{L}f(x,\varphi):=\mu_0f_x(x,\varphi)+\frac{1}{2}\left(\eta^{2 }f_{xx}(x,\varphi)+\gamma^2\varphi^{2}f_{\varphi\varphi}(x,\varphi)+2\gamma\eta\varphi f_{x\varphi}(x,\varphi)\right).
\end{equation*}
By the dynamic programming principle, we expect that $\overline{V}$ solves (in a suitable sense) the HJB equation (in the form of a variational inequality) 
\begin{equation}
\label{HJB}
\max\big\{(\rho-{\mathcal{L}}){u}(x,\varphi)-(1+\varphi)C(x), -u_{x}(x,\varphi)-K^+(1+\varphi), u_{x}(x,\varphi)-K^-(1+\varphi) \big\}=0,
\end{equation}
for $(x,\varphi)\in \R\times (0,\infty)$. 
In particular, we now first show that the value function $\overline{V}$ of the control problem defined in \eqref{eq:Vbar} is a viscosity solution to \eqref{HJB}; refer to \cite[Definition 4.5]{Federico2014} for the formal definition in a similar setting and references related to the validity of the dynamic programming principle. 
Following the arguments developed in \cite[Theorem 5.1, Section VIII.5]{FlemingSoner2006}, and using the a priori regularity obtained in Proposition \ref{prop:semiconc}, one can show the following classical result.

\begin{proposition}
\label{prop:Vvisc}
The value function $\overline{V}$ defined in \eqref{eq:Vbar} is a locally Lipschitz continuous viscosity solution to \eqref{HJB}.
\end{proposition}

Recall definition \eqref{C2} of the continuation region $\mathcal{C}_2$ of $\overline{v}(x,\varphi)$ in \eqref{eq:barv} and the relationship $\overline{V}_x(x,\varphi)=\overline{v}(x,\varphi)$ on $\R\times (0,\infty)$ from Proposition \ref{prop:Vbar}.$(ii)$, to see that  
\begin{equation}
\label{C2V}
{\mathcal{C}_2}=
\big\{(x,\varphi)\in\R\times (0,\infty): \ -K^+(1+\varphi)< \overline{V}_x(x,\varphi)<K^-(1+\varphi)\big\} .
\end{equation}
This implies that $\mathcal{C}_2$ identifies also with the so-called ``inaction region'' of $\overline{V}$, as suggested also by the HJB equation \eqref{HJB}. Combining the latter fact with Proposition \ref{prop:Vvisc} clearly implies the following result. 

\begin{corollary}
\label{cor:Vvisc}
The value function $\overline{V}$  defined in \eqref{eq:Vbar} is a locally Lipschitz continuous viscosity solution to 
$(\rho-{\mathcal{L}}){u}(x,\varphi)-(1+\varphi)C(x) = 0,$ for all $(x,\varphi) \in \mathcal{C}_2.$
\end{corollary}

The result in Corollary \ref{cor:Vvisc} will be used in the forthcoming analysis to upgrade the regularity of the value function in the closure of its inaction region which is the main goal of Section \ref{sec:HJB}. Before reaching this (final) step of our analysis in this section, we prove that $\overline{V}$ is actually globally continuously differentiable. 

\begin{proposition}
\label{prop:barVC1}
The value function in \eqref{eq:Vbar} satisfies 
$\overline{V}\in C^{1}(\R\times(0,\infty); \R)$.
\end{proposition}
\begin{proof} 
In order to prove that $\overline{V}\in C^{1}(\R\times(0,\infty); \R)$, we need to prove that both (classical) derivatives $\overline{V}_x(x,\varphi), \overline{V}_\varphi(x,\varphi)$ of $\overline{V}(x,\varphi)$ in the directions $x$ and $\varphi$, respectively, are continuous on $\R \times (0,\infty)$. We thus split the proof in two steps. 

\emph{Step 1. Continuity of $\overline{V}_x$.} 
We already know from Proposition \ref{prop:Vbar}.$(ii)$ that $\overline{V}_x=\bar{v}$ exists and from Proposition \ref{prop:barv}.$(i)$ that $(x,\varphi) \mapsto \bar{v}(x,\varphi)$ is continuous over $\R\times (0,\infty)$. Hence, we conclude that $(x,\varphi) \mapsto \overline{V}_x(x,\varphi)$ is continuous on $\R \times (0,\infty)$.

\emph{Step 2. Continuity of $\overline{V}_\varphi$.} 
Let us now show that the (classical) derivative $\overline{V}_{\varphi}$ exists at each $(x_{o},\varphi_{o})\in \R\times (0,\infty)$. 

We assume, without loss of generality\footnote{This can be done by replacing the (locally) semiconcave $\overline{V}(x,\varphi)$ by $W(x,\varphi):=\overline{V}(x,\varphi)-C_{0}|(x-x_{o},\varphi-\varphi_{o})|^{2}$ for suitable $C_0>0$ in the subsequent argument.}, that $\overline{V}$ is actually concave in a neighborhood $\mathcal{I}$ of $(x_{o},\varphi_{o})$.
Then, by concavity of $\overline{V}$ in $\mathcal{I}$, the right- and left-derivatives of $\overline{V}$ exist in the $\varphi$-direction at $(x_{o},\varphi_{o})$. 
We denote these derivatives by $\overline{V}_{\varphi}^{+}(x_{o},\varphi_{o})$ and $\overline{V}_{\varphi}^{-}(x_{o},\varphi_{o})$, respectively, and due to concavity they satisfy 
$\overline{V}_\varphi^{-}(x_{o},\varphi_{o}) \geq \overline{V}_{\varphi}^{+}(x_{o},\varphi_{o})$. 
Then, in order to show that $\overline{V}_{\varphi}$ exists, it suffices to show that the strict inequality $\overline{V}_\varphi^{-}(x_{o},\varphi_{o}) > \overline{V}_{\varphi}^{+}(x_{o},\varphi_{o})$ cannot hold. 
Aiming for a contradiction, we assume henceforth that $\overline{V}_\varphi^{-}(x_{o},\varphi_{o})>\overline{V}_{\varphi}^{+}(x_{o},\varphi_{o})$ does hold true.  

It follows from \cite[Theorem 23.4]{rock} and the fact that $\overline{V}_x$ exists and is continuous (cf.\ {\it Step 1} above) that there exist vectors
$$
{\zeta}:=(\overline{V}_x(x_{o},\varphi_{o}),\zeta_\varphi), \ {\eta}:=(\overline{V}_x(x_{o},\varphi_{o}),\eta_\varphi) \  
\in D^+\overline{V}(x_{o},\varphi_{o}) 
\quad \text{such that} \quad \zeta_\varphi < \eta_\varphi\,,
$$
where we denote by $D^+\overline{V}(x_{o},\varphi_{o})$ the superdifferential of $\overline{V}$ at $(x_{o},\varphi_{o})$. 
For any $(x,\varphi) \in \mathcal{I}$, we then define
$$g(x,\varphi):=\overline{V}(x_{o},\varphi_{o})+ \overline{V}_x(x_{o},\varphi_{o})(x-x_o)+\	\eta_{\varphi}(\varphi-\varphi_{o})\wedge \zeta_{\varphi}(\varphi-\varphi_{o}) $$
and notice that $\overline{V}(x_{o},\varphi_{o})=g(x_{o},\varphi_{o})$, while we also get by concavity that
$\overline{V}(x,\varphi)\leq g(x,\varphi)$, for all $(x,\varphi)\in\mathcal{I}$.
Next, we consider the sequence of functions $(f^{n})_{n\in\mathbb{N}}\subset C^2(\R\times (0,\infty); \R)$ defined by
$$
f^{n}(x,\varphi):= g(x,\varphi_o)+\tfrac{1}{2} (\eta_\varphi+\zeta_\varphi)(\varphi-\varphi_o)-\tfrac{n}{2} (\varphi-\varphi_o)^{2}, \quad \forall\; n\in\mathbb{N}.
$$
Such a sequence satisfies the following collection of properties, for any $n\in\mathbb{N}$:
\begin{equation*}
\begin{cases}
f^n(x_{o},\varphi_{o})=g(x_o,\varphi_o)=\overline{V}(x_{o},\varphi_{o}),  
\\  f^n\geq \overline{V}   \ \mbox{in a neighborhood of} \ (x_{o},\varphi_{o}),  
 \\  f^{n}_x(x_o,\varphi_{o})= \overline{V}_{x}(x_{o},\varphi_{o}), \ 
 f^{n}_{xx}(x_o,\varphi_{o})=0=f_{x\varphi}^{n}(x_o,\varphi_{o}), \ 
 f^{n}_{\varphi\varphi}(x_o,\varphi_{o})=-n .
\end{cases}
\end{equation*}
Then, using the viscosity subsolution property of $\overline{V}$ at $(x_o,\varphi_o)$ yields
$$
0\geq (\rho-{\mathcal{L}}){f^{n}}(x_{o},\varphi_{o})-(1+\varphi_{o})C(x_{o}) \stackrel{n\to\infty}{\longrightarrow } +\infty,
$$
which gives the desired contradiction. Hence, by arbitrariness of $(x_o,\varphi_o)$, we have that $\overline{V}$ is differentiable in the direction $\varphi$.

In view of the aforementioned differentiability in the direction $\varphi$ and the semiconcavity of $\overline{V}$ (cf.\ Proposition \ref{prop:semiconc}) we conclude from \cite[Theorem 25.5]{rock} that $\overline{V}_{\varphi}$ is continuous on $\R \times (0,\infty)$. 
\end{proof}

We are now ready to show the final result of this section, namely to upgrade the regularity of the control value function to the minimal required regularity for constructing a candidate optimal control policy and verify its optimality in Section \ref{sec:Verif}. 

To this end, we define for any $(x,\varphi) \in \R \times (0,\infty)$ the transformation
\begin{equation} \label{T}
T:=(T_1,T_2):\R\times(0,\infty)\to\R^2, \quad (T_1(x,\varphi),T_2(x,\varphi))=\big(x, x-\tfrac{\eta}{\gamma}\log(\varphi)\big),
\end{equation}
which is invertible with inverse given by
$T^{-1}(x,y)=(x,e^{\frac{\gamma}{\eta}(x-y)})$, for $(x,y)\in \R^2$.
Using the latter inverse transformation, we introduce the {\it transformed} version $\widehat{V}(x,y)$ of the value function $\overline{V}(x,\varphi)$ defined in \eqref{eq:Vbar} by  
\begin{equation}
\label{eq:Vhat}
\widehat{V}(x,y):=\overline{V} (x,e^{\frac{\gamma}{\eta} (x-y)}), \quad (x,y)\in\R^{2}.
\end{equation}
Moreover, direct calculations yield that
\begin{equation}
\label{Vhat-xy}
\widehat{V}_x(x,y) + \widehat{V}_y(x,y) = \overline{V}_x(x,e^{\frac{\gamma}{\eta} (x-y)}), \quad (x,y)\in\R^{2}.
\end{equation}
Given that $T:\R\times(0,\infty)\to\R^2$ is a global diffeomorphism, we have from \eqref{C2V} and \eqref{Vhat-xy} that the open set
\begin{equation}
\label{C3}
{\mathcal{C}_3} \hspace{-0.5mm}:= \hspace{-0.5mm}\big\{(x,y) \hspace{-0.5mm}\in \hspace{-0.5mm}\R^2 \hspace{-0.5mm}: \hspace{-0.5mm}-K^+ \hspace{-0.5mm}(1+e^{\frac{\gamma}{\eta}(x-y)}) \hspace{-0.5mm}< \hspace{-0.5mm}\big(\widehat{V}_x + \widehat{V}_y\big)(x,y) \hspace{-0.5mm}< \hspace{-0.5mm}K^- \hspace{-0.5mm}(1+e^{\frac{\gamma}{\eta}(x-y)}) \hspace{-0.5mm}\big\} \hspace{-0.5mm}= \hspace{-0.5mm}T({\mathcal{C}_2}).
\end{equation}

Finally, define the second-order linear differential operator on $f\in C^{2,1}(\R^2;\R)$ by 
\begin{equation}
\label{eq:LXY}
\mathcal{L}_{X,Y}f(x,y):= \tfrac{1}{2}\eta^{2}f_{{xx}}(x,y) + \mu_{0}f_{x}(x,y) + \tfrac{1}{2}(\mu_{0}+\mu_{1})f_{y}(x,y) 
\end{equation}

\begin{proposition}
\label{prop:regVhat}
The transformed value function $\widehat{V}$ defined in \eqref{eq:Vhat} is such that $\widehat{V}\in C^{1}(\R^{2};\R)$ and $\widehat{V}_{xx}\in L^{\infty}(\mathcal{C}_3;\R)$. In addition, $\widehat{V}$ is a classical solution to  
\begin{equation}
\label{eq:VhatinC3}
\big(\rho - \mathcal{L}_{X,Y}\big)u(x,y) =  C(x)(1+ e^{\frac{\gamma}{\eta} (x-y)}), 
\quad \text{for all}\,\, (x,y) \in {\mathcal{C}}_{3}. 
\end{equation}
\end{proposition}
\begin{proof}
First of all, due to Corollary \ref{cor:Vvisc} and the expression of the transformed value function in \eqref{eq:Vhat}, one can easily verify that $\widehat{V}$ is a viscosity solution to \eqref{eq:VhatinC3} on $\mathcal{C}_3$ due to \eqref{C3}. Then, in light of Proposition \ref{prop:barVC1} and the above smooth transformation, we also obtain that $\widehat{V}\in C^{1}(\R^{2};\R)$. 

By a standard localization argument based on the fact that $\widehat{V}$ is a continuously differentiable viscosity solution to \eqref{eq:VhatinC3} on $\mathcal{C}_{3}$ and results for Dirichlet boundary problems involving partial differential equations of parabolic type (see \cite{Liebermann}), we have that actually $\widehat{V}\in C^{2,1}(\mathcal{C}_{3};\R)$ and solves \eqref{eq:VhatinC3} on $\mathcal{C}_3$ in a classical sense. Hence, 
\begin{equation*}
\tfrac{1}{2}\eta^{2}\widehat{V}_{{xx}}(x,y) = 
- C(x)(1+ e^{\frac{\gamma}{\eta} (x-y)}) + \rho \widehat{V}(x,y) 
- \mu_{0}\widehat{V}_{x}(x,y) - \tfrac{1}{2}(\mu_{0}+\mu_{1})\widehat{V}_{y}(x,y),
\end{equation*}
for all $(x,y) \in \mathcal{C}_3$. However, since we know that $\widehat{V}\in C^{1}(\R^{2};\R)$ and the right-hand side of the above equation only involves continuous functions on $\R^2$, we conclude that $\widehat{V}_{{xx}}$ admits a continuous extension on $\overline{\mathcal{C}}_{3}$ (where $\overline{\mathcal{C}}_{3}$ denotes the closure of $\mathcal{C}_3$), so that $\widehat{V}_{xx}\in L^{\infty}(\mathcal{C}_3;\R)$. This completes the proof of the claim.
\end{proof}

\section{Verification Theorem and Optimal Control}
\label{sec:Verif}

Given the regularity of $\widehat{V}$ obtained in Proposition \ref{prop:regVhat} and the relation \eqref{eq:Vhat} between $\widehat{V}$ and $\overline{V}$ defined in \eqref{eq:Vbar}, we are now able to prove a verification theorem.  
Namely, in what follows, we provide the optimal control for $\overline{V}$ in terms of the boundaries $b_\pm$ defined in \eqref{eq:b+}. %--\eqref{eq:b-}. 
Before we commence the analysis, recall also the properties of $b_\pm$ proved in Proposition \ref{prop:barb}.

\subsection{Construction of control $\widehat{P}$ for state-space process $(X^{\widehat{P}},\Phi)$} 
\label{Control}

For any given $(x,\varphi)\in \R \times (0,\infty)$, we define the admissible control strategy $\widehat{P}:=\widehat{P}^+ - \widehat{P}^-$ such that the following couple of properties hold true $\Q$-a.s:
\begin{equation} 
\label{PXPhi0}
\begin{cases} 
b_{+}(\Phi_t) \leq X^{\widehat{P}}_t \leq b_{-}(\Phi_t), 
\ \text{for almost all}\,\, t\geq 0; \\
\displaystyle 
\widehat{P}^+_t = \scaleobj{.8}{\int_{[0,t]}} \mathds{1}_{\{X^{\widehat{P}}_{{s-}} \leq b_{+}(\Phi_s)\}} \d \widehat{P}^+_s , 
\quad \widehat{P}^-_t = \scaleobj{.8}{\int_{[0,t]}} \mathds{1}_{\{X^{\widehat{P}}_{{s-}} \geq b_{-}(\Phi_s)\}} \d \widehat{P}^-_s, \ \forall t \geq 0; \\
\displaystyle \scaleobj{.7}{\int_{0}^{\Delta \widehat{P}^+_t}} \mathds{1}_{\{(X^{\widehat{P}}_{t-}+ z, \Phi_t) \in \mathcal{C}_2\}} \d z + \scaleobj{.7}{\int_{0}^{\Delta \widehat{P}^-_t }} \mathds{1}_{\{(X^{\widehat{P}}_{t-} - z, \Phi_t) \in \mathcal{C}_2\}} \d z = 0, \ \forall t \geq 0, 
\end{cases}
\end{equation}
where $\Delta \widehat{P}^\pm_{t}:=\widehat{P}^\pm_{t}-\widehat{P}^\pm_{t-}$.

In practice, according to the aforementioned strategy, a lump-sum increase or decrease of the inventory process $X$ may be required, whenever the inventory level $X_{t-}$ happens to be either strictly below the boundary $b_+(\Phi_t)$ or above $b_-(\Phi_t)$, respectively. 
The purpose of these {\it jumps} of at most one of the controls $\widehat{P}_t^\pm$ at each such $t \geq 0$, of size either $(b_+(\Phi_t) - X^{\widehat{P}}_{t-})^+$ or $(X^{\widehat{P}}_{t-} - b_-(\Phi_t))^+$, is to bring immediately the inventory level $X_{t}$ inside the interval $[b_+(\Phi_{t}), b_-(\Phi_{t})]$. 
Mathematically, these are the actions caused at any time $t \geq 0$, by the jump parts $\Delta \widehat{P}^\pm_t$ of the controls $\widehat{P}^\pm$.
The strategy further prescribes taking action (increase or decrease the inventory) when the inventory process $X_t$ approaches, at any time $t \geq 0$, either boundary $b_+(\Phi_t)$ from above or $b_-(\Phi_t)$ from below. 
The purpose of these actions now is to make sure (with a minimal effort) that the inventory level $X_t$ is kept inside the interval $[b_+(\Phi_t), b_-(\Phi_t)]$. 
Mathematically, these actions are caused by the continuous parts of the respective controls $\widehat{P}^\pm$ and are the so-called {\it Skorokhod reflection-type} policies.

The nonincreasing property of $b_\pm(\cdot)$ (see Proposition \ref{prop:barb}.$(i)$) further implies that, the stronger the decision makers' belief is about a high average inventory level $\mu$ (i.e. higher $\varphi$, cf.\ \eqref{dynPhi0}), 
they tend to unload part of excess inventory more often so that inventory is kept below the optimal level $b_-(\varphi)$,
and delay placing replenishment orders by setting a lower optimal base-stock level $b_+(\varphi)$.

In multi-dimensional settings, the construction of a solution to a Skorokhod reflection problems is usually a delicate task, that is intimately related to the regularity of the reflection boundary (see \cite{DianettiFerrari} and \cite{Kruk} for a discussion and literature review). In our case, given that the dynamics of $X^{P}$ and $\Phi$ are decoupled and that $X^{P}=X^0 + P$ (cf.\ \eqref{eq:XPhiP}), the solution triplet $(X^{\widehat{P}}_t, \Phi_t, \widehat{P}_t)_{t\geq0}$ to the Skorokhod reflection problem at the boundaries $b_\pm$ can be constructed by adapting the iterative procedure developed in \cite[Section 4.3]{Federico2014}. In particular, with reference to the notation adopted in \cite{Federico2014}, we define
$\tau_0^+:=\inf\{t\geq0:\, x < b_+(\Phi_t) - \mu_0 t - \eta W_t\}$, 
$\tau_0^-:=\inf\{t\geq0:\, x > b_-(\Phi_t) - \mu_0 t - \eta W_t\}$ and 
$\tau_0:=\tau_0^+ \wedge \tau_0^-$.  
Notice that, because $\inf_{t\geq0}\big( b_-(\Phi_t) - b_+(\Phi_t) \big) >0$ by Proposition \ref{prop:barb}.(iii), we have $\{ \tau_0^+ = \tau_0^-\}=\{\tau_0=\infty\}$. Then, we set $\Omega_{\infty}:=\{\tau_0=\infty\}$, $\Omega_+:=\{\tau_0^+ < \tau_0^-\}$, $\Omega_-:=\{\tau_0^- < \tau_0^+\}$ and $C^0_t:= x$, for all $t\geq 0$, and recursively introduce: 
\begin{align*} 
\text{If $k\geq 1$ is odd,} \quad C^k_t &:=
\begin{cases} 
x, \,\, & \text{on}\,\, \Omega_{\infty}, \\
x + \max_{s \in [\tau_{k-1},t]}\big(b_+(\Phi_s) - \mu_0 s - \eta W_s - x )^+, \,\, & \text{on}\,\, \Omega_{+},\\
x + \min_{s \in [\tau_{k-1},t]}\big(b_-(\Phi_s) - \mu_0 s - \eta W_s - x )^-, \,\, & \text{on}\,\, \Omega_{-}, 
\end{cases} \\
%\end{align*}
%\begin{align*} 
\text{with} \quad \tau_k &:=
\begin{cases} 
\infty, \,\, & \text{on}\,\,\Omega_{\infty}, \\
\inf\{t\geq \tau_{k-1}:\, C^k_t > b_-(\Phi_t) - \mu_0 t -\eta W_t\},\,\, & \text{on}\,\, \Omega_{+},\\
\inf\{t\geq \tau_{k-1}:\, C^k_t < b_+(\Phi_t) - \mu_0 t -\eta W_t\},\,\, & \text{on}\,\, \Omega_{-}.
\end{cases} \\
\text{If $k\geq 2$ is even,} \quad C^k_t &:=
\begin{cases} 
x, \,\, & \text{on}\,\, \Omega_{\infty}, \\
x + \max_{s \in [\tau_{k-1},t]}\big(b_+(\Phi_s) - \mu_0 s - \eta W_s - x )^+, \,\, & \text{on}\,\, \Omega_{-},\\
x + \min_{s \in [\tau_{k-1},t]}\big(b_-(\Phi_s) - \mu_0 s - \eta W_s - x )^-, \,\, & \text{on}\,\, \Omega_{+},
\end{cases} \\
\text{with} \quad \tau_k &:=
\begin{cases} 
\infty, \,\, & \text{on}\,\, \Omega_{\infty}, \\
\inf\{t\geq \tau_{k-1}:\, C^k_t > b_-(\Phi_t) - \mu_0 t -\eta W_t\},\,\, & \text{on}\,\, \Omega_{-},\\
\inf\{t\geq \tau_{k-1}:\, C^k_t < b_+(\Phi_t) - \mu_0 t -\eta W_t\},\,\, & \text{on}\,\, \Omega_{+}.
\end{cases}
\end{align*}
In light of these definitions, one can then proceed as in \cite[Section 4.3]{Federico2014} in order to conclude the existence of a solution to the reflection problem \eqref{PXPhi0}.

It then follows from \eqref{PXPhi0} above together with the definitions \eqref{eq:b+} %--\eqref{eq:b-} 
of boundaries $b_\pm$, the region ${\mathcal{C}_2}$ from \eqref{C2b} and the fact that $\bar{v}=\overline{V}_x$ from Proposition \ref{prop:Vbar}.$(ii)$, that the nondecreasing processes $\widehat{P}^{\pm}$ are such that the state-space process $(X^{\widehat{P}},\Phi)$ and the induced (random) measures $\d \widehat{P}^{\pm}$ on $\R^+$ satisfy:
\begin{equation} \label{PXPhi}
\begin{cases} 
(X^{\widehat{P}}_t,\Phi_t) \in \overline{\mathcal{C}_2}, \quad \text{for $\Q \otimes \d t$-\text{a.e.}, with $\mathcal{C}_2$ as in \eqref{C2b}}; \\
\d \widehat{P}^{+} \text{ has support on } \big\{t\geq0:\, \overline{V}_x(X^{\widehat{P}}_t,\Phi_t) \leq -K^+(1+\Phi_t)\big\}; \\
\d \widehat{P}^{-} \text{ has support on } \big\{t\geq0:\, \overline{V}_x(X^{\widehat{P}}_t,\Phi_t) \geq K^-(1+\Phi_t)\big\} .
\end{cases}
\end{equation}

\subsection{Transformation of controlled process $(X^{\widehat{P}},\Phi)$ to $(X^{\widehat{P}},Y^{\widehat{P}})$} 
\label{TransY}

We now use the transformation \eqref{T} from $(x,\varphi)$- to $(x,y)$-coordinates, in order to define the controlled process 
\begin{equation}
\label{Y=XPhi}
Y^{\widehat{P}}_t:=X^{\widehat{P}}_t - \tfrac{\eta}{\gamma}\log(\Phi_t), \quad t\geq0.
\end{equation}
Recalling the transformed value function \eqref{eq:Vhat} and the relation in \eqref{Vhat-xy}, we have  
\begin{equation*}
\widehat{V}(X^{\widehat{P}}_t, Y^{\widehat{P}}_t):=\overline{V} \big(X^{\widehat{P}}_t,e^{\frac{\gamma}{\eta} (X^{\widehat{P}}_t-Y^{\widehat{P}}_t)} \big) 
, \; 
(\widehat{V}_x + \widehat{V}_y) (X^{\widehat{P}}_t, Y^{\widehat{P}}_t) 
= \overline{V}_x\big(X^{\widehat{P}}_t ,e^{\frac{\gamma}{\eta} (X^{\widehat{P}}_t-Y^{\widehat{P}}_t)}\big) ,
\end{equation*}
under the dynamics
\begin{equation}
\label{eq:XYP}
\begin{cases}
\d X^{\widehat{P}}_t = \mu_0\d t + \eta \d W_t+\d {\widehat{P}}_t^+-\d {\widehat{P}}_t^-, \quad &X^{\widehat{P}}_{0-}=x \in \R,\\
\d Y^{\widehat{P}}_t = \frac{1}{2}(\mu_0+\mu_1) \d t + \d \widehat{P}^+_t - \d \widehat{P}^-_t, \quad &Y^{\widehat{P}}_{0-}= y := x - \frac{\eta}{\gamma}\log(\varphi) \in \R.
\end{cases}
\end{equation}
Hence, we can express the control $\widehat{P}$ defined in Section \ref{Control} in terms of the state-space process $(X^{\widehat{P}},Y^{\widehat{P}})$ via 
\begin{equation} \label{PXY}
\begin{cases}
(X^{\widehat{P}}_t,Y^{\widehat{P}}_t) \in \overline{\mathcal{C}_3}, \quad \text{for $\Q \otimes \d t$-\text{a.e.}, where $\mathcal{C}_3$ is defined in \eqref{C3}}; \\
\d \widehat{P}^{+} \text{ has support on } \big\{t\geq0: \big(\widehat{V}_x + \widehat{V}_y\big)(X^{\widehat{P}}_t,Y^{\widehat{P}}_t) \leq -K^+\big(1+e^{\frac{\gamma}{\eta} (X^{\widehat{P}}_t-Y^{\widehat{P}}_t)}\big)\big\}; \\ 
\d \widehat{P}^{-} \text{ has support on } \big\{t\geq0: \big(\widehat{V}_x + \widehat{V}_y\big)(X^{\widehat{P}}_t,Y^{\widehat{P}}_t) \geq K^-\big(1+e^{\frac{\gamma}{\eta} (X^{\widehat{P}}_t-Y^{\widehat{P}}_t)})\big)\big\}.
\end{cases} 
\end{equation}

\subsection{Optimality of control $\widehat{P}$} 
\label{OptControl}

In this section we prove the optimality of the control $\widehat{P}$ defined through \eqref{PXPhi0}, which is equivalently expressed by \eqref{PXPhi} in terms of the state-space process $(X^{\widehat{P}},\Phi)$ and by \eqref{PXY} in terms of the state-space process $(X^{\widehat{P}},Y^{\widehat{P}})$, see Sections \ref{Control}--\ref{TransY}.

\begin{theorem}[Verification Theorem]
\label{thm:Verif}
The admissible control $\widehat{P}\in{\mathcal{A}}$ defined through \eqref{PXPhi0} (see also \eqref{PXPhi} and \eqref{PXY}) is optimal for Problem \eqref{eq:Vbar}. Actually, $\widehat{P}$ is the unique optimal control (up to indistinguishability) if $C$ is strictly convex.
\end{theorem}

\begin{proof}
Let $(X^{\widehat{P}}_{0-},Y^{\widehat{P}}_{0-})=(x,y) \equiv (x, x -{\eta}\log(\varphi)/\gamma) \in \overline{\mathcal{C}_3}$ be given and fixed. 
Define $\tau_n:=\inf \{t \geq 0: |(X^{\widehat{P}}_t,Y^{\widehat{P}}_t)| > n \} \wedge n$, for $n\in \mathbb{N}$, with state-space process $(X^{\widehat{P}},Y^{\widehat{P}})$ as in \eqref{eq:XYP}, and recall that $(X^{\widehat{P}}_t,Y^{\widehat{P}}_t) \in \overline{\mathcal{C}_3}$, $\Q$-a.s.\ for all $t\geq 0$. 
In particular, Lemma \ref{lemma-app} in Appendix \ref{C} yields that for any $t\geq0$, $\Q\big((X^{\widehat{P}}_t,\Phi_t) \in \mathcal{C}_2\big)=1$, and therefore $\Q\big((X^{\widehat{P}}_t,Y^{\widehat{P}}_t) \in \mathcal{C}_3\big)=1$.
Then, given the regularity of $\widehat{V}$ (cf.\ Proposition \ref{prop:regVhat}), we can employ the approximation argument via mollifiers developed in the proof of \cite[Theorem 4.1, Chapter VIII]{FlemingSoner2006}, in order to conclude that
\begin{align*}
&\widehat{V}(x,y) = \E^{\Q}\bigg[e^{-\rho \tau_n}\widehat{V}(X^{\widehat{P}}_{\tau_n},Y^{\widehat{P}}_{\tau_n})\bigg] - \E^{\Q}\bigg[\int_{0}^{\tau_n} \hspace{-3mm}e^{-\rho s}\big(\mathcal{L}_{X,Y}-\rho\big)\widehat{V}(X^{\widehat{P}}_s,Y^{\widehat{P}}_s) \d s\bigg] \nonumber \\
& - \E^{\Q}\bigg[\int_{0}^{\tau_n} \hspace{-3mm}e^{-\rho s} \big(\widehat{V}_x+ \widehat{V}_y\big)(X^{\widehat{P}}_s,Y^{\widehat{P}}_s) \d \widehat{P}^c_s 
- \sum_{0\leq s \leq \tau_n} \hspace{-3mm}e^{-\rho s}\Big(\widehat{V}(X^{\widehat{P}}_s,Y^{\widehat{P}}_s) - \widehat{V}(X^{\widehat{P}}_{s-},Y^{\widehat{P}}_{s-})\Big)\bigg], 
\end{align*}
where $\widehat{P}^c$ denotes the continuous part of $\widehat{P}$ and the final sum is non-zero only for (at most countably many) times $s$ such that
$\Delta \widehat{P}_s:=\widehat{P}_s - \widehat{P}_{s-} \neq 0$. 
Clearly, $\Delta \widehat{P}_s = \Delta \widehat{P}^+_s - \Delta \widehat{P}^-_s$, 
where $\Delta \widehat{P}^{\pm}_s:=\widehat{P}^{\pm}_s - \widehat{P}^{\pm}_{s-}$ 
and notice that 
\begin{align*} 
\sum_{0\leq s \leq \tau_n} \hspace{-3mm}e^{-\rho s} \bigg\{ 
\big(\widehat{V}(X^{\widehat{P}}_s,Y^{\widehat{P}}_s) - \widehat{V}(X^{\widehat{P}}_{s-},&Y^{\widehat{P}}_{s-})\big) 
- \scaleobj{0.8}{\int_0^{\Delta \widehat{P}^+_s}} \hspace{-1mm}\big(\widehat{V}_x + \widehat{V}_y\big)(X^{\widehat{P}}_{s-} + u, Y^{\widehat{P}}_{s-} + u) \d u \nonumber\\ 
&+ \scaleobj{0.8}{\int_0^{\Delta \widehat{P}^-_s}} \hspace{-1mm}\big(\widehat{V}_x + \widehat{V}_y\big)(X^{\widehat{P}}_{s-} - u, Y^{\widehat{P}}_{s-} - u) \d u \bigg\} = 0.
\end{align*}
Hence, plugging the last formula into the penultimate one and using \eqref{eq:VhatinC3}, the nonnegativity of $\widehat{V}$, the second and third property of control $\widehat{P}$ in \eqref{PXY}, we see that 
\begin{align}
&\widehat{V}(x,y) 
\geq \E^{\Q}\bigg[\int_{0}^{\tau_n}e^{-\rho s} \big(1+e^{\frac{\gamma}{\eta} (X^{\widehat{P}}_s-Y^{\widehat{P}}_s)}\big) C(X^{\widehat{P}}_s) \d s \bigg] \nonumber \\
& + \E^{\Q}\bigg[\int_{0}^{\tau_n}e^{-\rho s} K^+ \big(1+e^{\frac{\gamma}{\eta} (X^{\widehat{P}}_s-Y^{\widehat{P}}_s)}\big) \d \widehat{P}^{+}_s + \int_{0}^{\tau_n}e^{-\rho s} K^- \big(1+e^{\frac{\gamma}{\eta} (X^{\widehat{P}}_s-Y^{\widehat{P}}_s)}\big) \d \widehat{P}^{-}_s\bigg] . \nonumber
\end{align}
Then, we take limits as $n\uparrow \infty$ and we invoke Fatou's lemma (given the nonnegativity of all the integrands above) to find that
\begin{align*}
&\widehat{V}(x,y) 
\geq \E^{\Q}\bigg[\int_{0}^{\infty}e^{-\rho s} \big(1+e^{\frac{\gamma}{\eta} (X^{\widehat{P}}_s-Y^{\widehat{P}}_s)}\big) C(X^{\widehat{P}}_s) \d s \bigg] 
\nonumber\\
&+ \E^{\Q}\bigg[\int_{0}^{\infty}e^{-\rho s} K^+ \big(1+e^{\frac{\gamma}{\eta} (X^{\widehat{P}}_s-Y^{\widehat{P}}_s)}\big) \d \widehat{P}^{+}_s + \int_{0}^{\infty}e^{-\rho s} K^- \big(1+e^{\frac{\gamma}{\eta} (X^{\widehat{P}}_s-Y^{\widehat{P}}_s)}\big) \d \widehat{P}^{-}_s\bigg] .
\end{align*}
Given now that $X^{\widehat{P}} - Y^{\widehat{P}} = {\eta}\log(\Phi)/{\gamma}$ by definition \eqref{Y=XPhi}, and that \eqref{eq:Vhat} yields $\widehat{V}(x,y) = \widehat{V}(x,x - {\eta}\log(\varphi)/{\gamma}) = \overline{V}(x,\varphi)$, we further conclude from the latter inequality that for any $(x,\varphi) \in \overline{\mathcal{C}_2}$ (as we assumed $(x,y) \equiv (x, x -{\eta}\log(\varphi)/{\gamma}) \in \overline{\mathcal{C}_3}$)
\begin{equation}
\label{eq:JhatVhat}
\overline{V}(x,\varphi) \geq \E^{\Q}\bigg[\int_{0}^{\infty} \hspace{-3mm}e^{-\rho s} \big(1 + \Phi_s\big) C(X^{\widehat{P}}_s) \d s + \int_{0}^{\infty} \hspace{-3mm}e^{-\rho s} \big(1 + \Phi_s\big)\big(K^+ \d \widehat{P}^{+}_s \hspace{-0.5mm}+ \hspace{-0.5mm}K^- \d \widehat{P}^{-}_s\big)\bigg]. 
\end{equation}
Combining this inequality with definition \eqref{eq:Vbar}, i.e.\ $\overline{V}(x,\varphi) \leq \overline{\mathcal{J}}_{x,\varphi}(\widehat{P})$, we prove that $\widehat{P}$ is an optimal control, for any $(x,\varphi) \in \overline{\mathcal{C}_2}$.

Suppose now that $(x,\varphi)$ is such that $x < b_+(\varphi)$, so that $(x,\varphi) \in \mathcal{S}^+_2$. Then, according to \eqref{PXPhi0} (see also \eqref{PXPhi}), and using \eqref{eq:JhatVhat}, we have that
\begin{align*}
\overline{\mathcal{J}}_{x,\varphi}(\widehat{P}) 
&= K^+(1+\varphi)\big(b_+(\varphi) - x) + \overline{\mathcal{J}}_{b_+(\varphi),\varphi}(\widehat{P}) \\
&\leq \overline{V}(b_+(\varphi),\varphi) - \int_x^{b_+(\varphi)} \overline{V}_x(z,\varphi) = \overline{V}(x,\varphi).
\end{align*}
Proceeding similarly also for $(x,\varphi)$ such that $x > b_-(\varphi)$, we conclude that $\widehat{P}$ is indeed optimal for any $(x,\varphi) \in \R^2$. 
\end{proof}

\section{Refined Regularity of the Free Boundaries and their Characterization}
\label{sec:FBchar}

In this section we will obtain substantial regularity of the value $\bar v(x,\varphi)$ of the Dynkin game \eqref{eq:barv}, as well as an analytical characterisation of its corresponding free boundaries $b_\pm$, and consequently the optimal control rule $\widehat{P}$ (see Theorem \ref{thm:Verif}).

\subsection{Parabolic formulation and Lipschitz continuity of the free boundaries} 
\label{sec:ChangeVar}

In view of a further change of variables, in line with \eqref{Y=XPhi}, we define
$Y^0_t:=X^0_t - \frac{\eta}{\gamma}\log(\Phi_t)$, $t \geq 0$,
with $X^0$ as in \eqref{XoPhi}. Then, by It\^o's formula, we have
\begin{equation} \label{X0Y0}
\begin{cases}
\d X^{0}_t = \mu_0\d t + \eta \d W_t , \quad &X^{0}_{0}=x \in \R,\\
\d Y^{0}_t = \frac{1}{2}(\mu_0+\mu_1) \d t , \quad &Y^{0}_{0}= y := x - \frac{\eta}{\gamma}\log(\varphi) \in \R,
\end{cases}
\end{equation}
and \eqref{eq:barv} rewrites in terms of the new coordinates $(x,y)=(X^0_0,Y^0_0)$ as 
\begin{align}
\label{eq:vhat}
\hat{v}(x,y):= \inf_{\sigma}\sup_{\tau} \E^{\Q}
&\bigg[\int_0^{\tau\wedge\sigma} \hspace{-3mm}e^{-\rho t}\Big(1+e^{\frac{\gamma}{\eta}(X^0_t-Y_t)}\Big) C'(X_t^0)\d t 
- e^{-\rho \tau} \Big(1+e^{\frac{\gamma}{\eta}(X^0_\tau-Y_\tau)}\Big) \times  \nonumber \\ 
&K^+ \mathds{1}_{\{\tau<\sigma\}}
+e^{-\rho \sigma}\Big(1+e^{\frac{\gamma}{\eta}(X_\sigma^0 -Y_\sigma)}\Big)  K^-\mathds{1}_{\{\tau>\sigma\}}\bigg] 
= \bar{v}\left(x,e^{\frac{\gamma}{\eta}(x-y)}\right) 
\end{align}
for $(x,y) \in \R^2$.
In view of the relationship in \eqref{eq:vhat}, the value function $\hat{v}(\cdot,\cdot)$ inherits important properties which have already been proved for $\bar{v}(\cdot,\cdot)$.  To be more precise, we first conclude immediately from Proposition \ref{prop:barv}.$(i)$ the following result.
\begin{proposition}
\label{prop:hatv}
The value function $(x,y) \mapsto \hat{v}(x,y)$ defined in \eqref{eq:vhat} is continuous over $\R^2$.
\end{proposition}

Moreover, since $\bar{v}(x,\exp\{{\gamma}(x-y)/{\eta}\})= \overline{V}_x(x,\exp\{{\gamma}(x-y)/{\eta}\})$ by Proposition \ref{prop:Vbar}.$(ii)$, it follows from \eqref{Vhat-xy} that $\hat{v}(x,y) = \widehat{V}_x(x,y) + \widehat{V}_y(x,y)$ for all $(x,y) \in \R^2$, and consequently the open set ${\mathcal{C}_3}$ defined in \eqref{C3} takes the form 
\begin{align}
\label{C3hat}
{\mathcal{C}_3}= \hspace{-1mm}\big\{(x,y)\in\R^2: -K^+\big(1+e^{\frac{\gamma}{\eta}(x-y)}\big) < \hat{v}(x,y)<K^- \big(1+e^{\frac{\gamma}{\eta}(x-y)}\big) \big\} \hspace{-1mm}= T(\mathcal{C}_2).
\end{align}
Hence, by also defining the closed sets
\begin{align}
\label{S3+}
\begin{split}
{\mathcal{S}_3^+}&:=\big\{(x,y)\in\R^2: \ \hat{v}(x,y)\leq - K^+\big(1+e^{\frac{\gamma}{\eta}(x-y)}\big) \big\}, \\
%\label{S3-}
{\mathcal{S}_3^-}&:=\big\{(x,y)\in\R^2: \ \hat{v}(x,y)\geq K^- \big(1+e^{\frac{\gamma}{\eta}(x-y)}\big) \big\},
\end{split}
\end{align}
the global diffeomorphism $T$ from \eqref{T} implies that ${\mathcal{S}_3^\pm}=T({\mathcal{S}_2^{\pm}})$ as well, where $\mathcal{C}_2$ and $\mathcal{S}_2^{\pm}$ are the continuation and stopping regions \eqref{C2}--\eqref{S2} for the Dynkin game $\bar{v}$ in \eqref{eq:barv}. 
Combining these relationships with the structure of the latter regions  in \eqref{C2b} %--\eqref{S2b} 
yields that ${\mathcal{C}_3}$ and ${\mathcal{S}_3}^{\pm}$ are connected.

In order to obtain the explicit structure of the regions ${\mathcal{C}_3}$ and ${\mathcal{S}_3}^{\pm}$, we now define the generalised inverses of the nonincreasing $b_{\pm}$ (cf.\ Proposition \ref{prop:barb}) by
\begin{equation}
\label{bpminv}
b_+^{-1}(x) \hspace{-0.5mm}:= \hspace{-0.5mm}\sup\{\varphi \hspace{-0.5mm}\in \hspace{-0.5mm}(0,\infty) \hspace{-0.5mm}: \hspace{-0.5mm}b_+(\varphi) \hspace{-0.5mm}\geq \hspace{-0.5mm}x\}, \quad b_-^{-1}(x) \hspace{-0.5mm}:= \hspace{-0.5mm}\inf\{\varphi \hspace{-0.5mm}\in \hspace{-0.5mm}(0,\infty) \hspace{-0.5mm}: \hspace{-0.5mm}b_-(\varphi) \hspace{-0.5mm}\leq \hspace{-0.5mm}x\}.
\end{equation}
Since the map $\varphi\mapsto T_2(x,\varphi)$ in \eqref{T} is decreasing for any given $x\in\R$ (cf.\ the functions $b_{\pm}$ are nonincreasing due to Proposition \ref{prop:barb}.$(i)$), we have
\begin{eqnarray*}
(x,y)\in{\mathcal{C}_3} 
& \Leftrightarrow & \big(x, e^{\frac{\gamma}{\eta}(x-y)}\big) \in {\mathcal{C}_2} 
\Leftrightarrow \ x - \tfrac{\eta}{\gamma}\log (b^{-1}_{-}(x)) < y< x - \tfrac{\eta}{\gamma}\log (b^{-1}_{+}(x)) ,
\end{eqnarray*}
while similar relations hold true for the characterisation of ${\mathcal{S}_3^\pm}$.
Then, by defining
\begin{equation}
\label{Cpm}
c^{-1}_{\pm}(x):= x-\tfrac{\eta}{\gamma}\log (b^{-1}_{\pm}(x)),
\end{equation}
we can obtain the structure of the continuation and stopping regions of $\widehat{v}$, as
\begin{align}
\label{C3bis}
\begin{split}
&{\mathcal{C}_3}=\{(x,y)\in\R^2: \ c_-^{-1}(x)<y<c_+^{-1}(x)\},
\\
%\label{S3pmbis}
{\mathcal{S}_3^+}=\{(x,y)\in\;&\R^2: \ y\geq c_+^{-1}(x)\} 
\quad \text{and} \quad 
{\mathcal{S}_3^-}=\{(x,y)\in\R^2: \ y\leq c_-^{-1}(x)\}.
\end{split}
\end{align}

The next lemma can be proved thanks to \eqref{bpminv}, \eqref{Cpm} and Proposition \ref{prop:barb}. 

\begin{lemma} \label{c-1}
The functions $c^{-1}_\pm(\cdot)$ defined in \eqref{Cpm} are strictly increasing, while $c^{-1}_+(\cdot)$ is left-continuous and $c^{-1}_-(\cdot)$ is right-continuous on $\R$.
\end{lemma}

In light of Lemma \ref{c-1}, for $y \in \R$, we may define the functions   
\begin{equation}
\label{eq:cpm}
c_{+}(y):=\inf\{x\in\R: \ y\leq c^{-1}_+(x)\} 
\quad \text{and} \quad 
c_-(y):=\sup\{x\in\R: \ y\geq c^{-1}_-(x)\} .
\end{equation}
In the following result, we prove that $y \mapsto c_\pm(y)$ identify with the optimal free boundaries of the Dynkin game $\hat{v}$ in \eqref {eq:vhat} and provide some important properties such as their global Lipschitz continuity. 

\begin{proposition} \label{prop:propcpm}
The free boundaries $c_\pm$ defined in \eqref{eq:cpm}. Then, 
\begin{enumerate}
\item[(i)] $c_{\pm}(\cdot)$ are nondecreasing on $\R$ and we have $x_+^*\leq c_+(y)<c_-(y)\leq x_-^*$ for all $y\in\R$ (with $x^*_{\pm}$ as in Proposition \ref{prop:aa}). Moreover, $c_+(y) \leq (C')^{-1}(-\rho K^+)$ and $c_-(y) \geq (C')^{-1}(\rho K^-)$ for all $y \in \R$;
\item[(ii)] $c_\pm(\cdot)$ are Lipschitz-continuous on $\R$ with Lipschitz constant $L=1$, namely 
$0\leq c_\pm(y)-c_\pm(y')\leq y-y'$, for all $y\geq y'$.
\item[(iii)] The structure of the continuation and stopping regions for \eqref{eq:vhat} take the form
\begin{align*}
&{\mathcal{C}_3}=\{(x,y)\in\R^2: \ c_+(y)<x<c_-(y)\}, \\
{\mathcal{S}_3^+}=\{(x,y)\in \;&\R^2: \ x \leq c_+(y)\}
\quad \text{and} \quad 
{\mathcal{S}_3^-}=\{(x,y)\in\R^2: \ x \geq c_-(y)\}.
\end{align*}
\end{enumerate}
\end{proposition}
\begin{proof}
{\it Proof of (i).} The first part of the claim follows from Lemma \ref{c-1}, together with the definition \eqref{eq:cpm} of $c_{\pm}$. The second and third parts of the claim are due to the fact that $T_1$ as in \eqref{T} is the identity.

{\it Proof of (ii).} Using the definitions \eqref{Cpm} of $c^{-1}_\pm$ and the monotonicity of $b^{-1}_\pm$ (see proof of Lemma \ref{c-1}) we get
\begin{equation}
\label{est}
c^{-1}_{\pm}(x)-c^{-1}_{\pm}(x') 
= x - \tfrac{\eta}{\gamma}\log (b^{-1}_{\pm}(x)) - x' + \tfrac{\eta}{\gamma}\log (b^{-1}_{\pm}(x')) 
\geq x-x',  \ \ \ \forall\; x\geq x'.
\end{equation}
Combining this with definitions \eqref{eq:cpm} and part $(i)$, we obtain the desired claim. 

{\it Proof of (iii).} This is again due to the definitions \eqref{eq:cpm} of $c_{\pm}$, their monotonicity from part $(i)$ and the expressions of the sets in \eqref{C3bis}. % and \eqref{S3pmbis}.
\end{proof}

\subsection{Global $C^1$-regularity of $\widehat{v}$}
\label{sec:smoothfit}

For any $(x,y) \in \R^2$ given and fixed, we consider the strong solution to the dynamics in \eqref{X0Y0}, denoted by 
$X_t^{0,x}=x + \mu_0 t + \eta {W}_t$ and $Y^{0,y}_t=y + \frac{1}{2}(\mu_1+\mu_0) t$,  $t\geq0$
and we define 
\begin{equation}\label{star}
\tau^{\star}\hspace{-0.5mm}(x,y) \hspace{-0.5mm}:= \hspace{-0.5mm}\inf\{t\geq 0 \hspace{-0.5mm}: \hspace{-0.5mm}(X_t^{0,x},Y_t^{0,y}) \hspace{-0.5mm}\in \hspace{-0.5mm}\mathcal{S}_3^{+}\}, \ 
\sigma^{\star}\hspace{-0.5mm}(x,y) \hspace{-0.5mm}:= \hspace{-0.5mm}\inf \{t\geq 0 \hspace{-0.5mm}: \hspace{-0.5mm}(X_t^{0,x},Y_t^{0,y}) \hspace{-0.5mm}\in \hspace{-0.5mm}\mathcal{S}_3^{-}\}.
\end{equation}

Notice that, in light of the one-to-one and onto transformations $\overline{T}$ and $T$, the pair $(\tau^{\star}(x,y),\sigma^{\star}(x,y))$ realises a saddle point for the Dynkin game with value $\hat{v}(x,y)$ in \eqref{eq:vhat} if and only if, by setting $\pi:=e^{\frac{\gamma}{\eta}(x-y)}/(1+e^{\frac{\gamma}{\eta}(x-y)})$, the stopping times $\widetilde{\tau}(x,\pi):=\inf\{t\geq 0: (X_t^{0,x},\Pi_t^{\pi}) \in \mathcal{S}_1^{+}\}$ and $\widetilde{\sigma}(x,\pi):=\inf\{t\geq 0: (X_t^{0,x},\Pi_t^{\pi}) \in \mathcal{S}_1^{-}\}$ form a saddle point for the game with value $v(x,\pi)$ in \eqref{optstop}. 
In order to prove the latter claim, one can apply \cite[Theorem 2.1]{EkstromPes} upon setting (in their notation) $X_t:=(t, X^0_t, \Pi_t)$, $G_1(t,x,\pi):=e^{-\rho t}(-K_+ - M(x,\pi))$, $G_2(t,x,\pi):=e^{-\rho t}(K_- - M(x,\pi))$, $G_3(t,x,\pi):=0$, with $M(x,\pi):=\E_{(x,\pi)}[\int_0^{\infty} e^{-\rho t} C'(X^0_t) \d t]$, and noticing that $\E_{(x,\pi)}[\sup_{t\geq0} e^{-\rho t}|M(X^0_t,\Pi_t)|]< \infty$. 
This follows via \eqref{X0Pi} and standard estimates employing Assumption \ref{ass:C}, which yield that $|M(x,\pi)| \leq \kappa (1+|x|^{p-1})$.

In the sequel, we aim at deriving the global $C^1$-regularity of $\hat{v}(\cdot,\cdot)$. %, following the arguments developed in \cite{DeAPes}.
In order to accomplish that, we need the following result about the regularity (in the probabilistic sense) of $(\tau^{\star},\sigma^{\star})$. 
%Its proof can be obtained by combining the global Lipschitz property of $c_{\pm}$ and the law of iterated logarithm of Brownian motion.
%; refer to \cite{FFRArxiv} for full details. 

\begin{lemma}
\label{lemmasf}
Suppose that $(x_n,y_n)_{n\in\mathbb{N}^*}\subset \mathcal{C}_3$ is such that $(x_n,y_n)\to (x_o,y_o)$, where $y_o\in\R$ and $x_o:=c_+(y_o)$ (resp.,\  $x_o:=c_-(y_o)$), then $\tau^{\star}(x_n,y_n)\to 0$ (resp.,\ $\sigma^{\star}(x_n,y_n)\to 0$), $\Q$-a.s..
\end{lemma}

\begin{proof}
We prove the claim for $\tau^{\star}(x_n,y_n)$, since the proof for $\sigma^{\star}(x_n,y_n)$ can be performed analogously. Fix $\omega\in\Omega$ and assume (aiming for a contradiction) that 
$
\limsup_{n\rightarrow \infty} \tau^{\star}(x_n,y_n)(\omega)=:\delta>0.
$
Namely, there exists a subsequence, still labelled by $(x_n,y_n)$, such that 
$X_t^{0,x_n}(\omega) >  c_+(Y^{0,y_n}_t)$, for all $n\in\mathbb{N}^*$ and $t\in[0,\delta/2]$, 
that is,
\begin{equation}\label{strict}
x_n +\mu_0 t + \eta {W}_t(\omega)>  c_+\big(y_n + \tfrac{1}{2}(\mu_1+\mu_0) t\big) \qquad \forall\; n\in\mathbb{N}^*,  \quad \forall\; t\in[0,\delta/2]. 
\end{equation}
Hence, taking the limit as $n\to\infty$ and considering that $c_+$ is continuous (see Proposition \ref{prop:propcpm}.$(ii)$),
$\eta {W}_t(\omega) \geq c_+ (y_o + \frac{1}{2}(\mu_1+\mu_0) t) - x_o - \mu_0 t$, 
for all $t\in[0,\delta/2].$
Using now the Lipschitz continuity of $c_+$ (see again Proposition \ref{prop:propcpm}.$(ii)$), we further obtain $\forall\; n\in\mathbb{N}^*$ and $\forall\; t\in[0,\delta/2]$ that
\begin{align}
\label{eq:LIL}
\eta {W}_t(\omega) & \geq c_+(y_o) - \tfrac{1}{2}(\mu_1+\mu_0)^- t - x_o - \mu_0 t 
\ = \  - \tfrac{1}{2}\big( (\mu_1+\mu_0)^- + \mu_0 \big) t. 
\end{align}
However, by the law of iterated logarithm, we have that \eqref{eq:LIL} can only happen for $\omega$ belonging to a $\Q$-null set and the proof is complete.
\end{proof}

\begin{remark} 
\label{rem:check}
From the previous proof one can easily observe that, by replacing the strict inequality with the large one in \eqref{strict}, we can actually prove that $\check\tau^{\star}(x_n,y_n)\to 0$ and $\check\sigma^{\star}(x_n,y_n)\to 0$, $\Q$-a.s., where
\begin{align}\label{check}
\check{\tau}^{\star}(x,y)&:=\inf\{t\geq 0: (X_t^{0,x},Y_t^{0,y})\in 	\mathrm{Int}(\mathcal{S}_3^+)\}, \\
\label{check-sigma}
\check{\sigma}^{\star}(x,y)&:=\inf\{t\geq 0: (X_t^{0,x},Y_t^{0,y})\in 	\mathrm{Int}(\mathcal{S}_3^-)\}.
\end{align}
\end{remark}

We now show that the value function $\hat{v}(x,y)$ of the Dynkin game \eqref{eq:vhat} is smooth across the topological boundary $\partial \mathcal{C}_3$ of the continuation region $\mathcal{C}_3$ from \eqref{C3hat} in both directions $x$ and $y$. 
The proof exploits the probabilistic expressions of the derivatives of $\hat{v}$, Lemma \ref{lemmasf} and Remark \ref{rem:check}. 
%Full details can be found in the extended version of this paper \cite{FFRArxiv}.}
%%%%%%%%
Its proof can be found in the Appendix \ref{B}. 
%%%%%%%%

\begin{proposition}[Smooth-fit]
\label{prop:sf}
Let $y_o\in \R$ and set $x_o:=c_\pm(y_o)$. Then the value function $\hat{v}$ defined in \eqref{eq:vhat} satisfies
$$\lim_{\stackrel{(x,y)\to(x_o,y_o)}{(x,y) \in \mathcal{C}_3}} \hat{v}_{x}(x,y) = {\mp} \frac{\gamma}{\eta}K^\pm e^{\frac{\gamma}{\eta}(x_o-y_o)} 
, \quad \lim_{\stackrel{(x,y)\to(x_o,y_o)}{(x,y)\in\mathcal{C}_3}} \hat{v}_{y}(x,y)
= {\pm} \frac{\gamma}{\eta}K^\pm e^{\frac{\gamma}{\eta}(x_o-y_o)}.$$ 
\end{proposition}

We are now ready to derive the global $C^1$-regularity of $\hat{v}$ as well as the local boundedness of its second derivative in $x$.
 
\begin{proposition}
\label{prop:reg}
The value function $\hat{v}$ defined in \eqref{eq:vhat} satisfies $\widehat{v}\in C^{1}(\R^2;\R)$ and $\widehat{v}_{xx}\in L^\infty_{\text{loc}}(\R^2;\R)$.
\end{proposition}
\begin{proof}
By standard arguments based on the strong Markov property and Dirichlet boundary problems involving second-order partial differential equations of parabolic type, one can show that $\hat{v}$ in \eqref{eq:vhat} is a classical $C^{2,1}$-solution to
$(\rho-\mathcal{L}_{X,Y})u(x,y) - \big(1+e^{\frac{\gamma}\eta(x-y)}\big) C'(x)=0$, for all $(x,y) \in {\mathcal{C}}_{3},$
where $\mathcal{L}_{X,Y}$ is the second-order differential operator defined in \eqref{eq:LXY} and $\mathcal{C}_{3}$ is given by \eqref{C3hat} (see also Proposition \ref{prop:propcpm}.$(iii)$). Also, $\hat{v} \in C^{\infty}$ in the interior of $\mathcal{S}^{\pm}_3$. Hence, by Proposition \ref{prop:sf} we have that $\widehat{v}\in C^{1}(\R^2;\R)$. 

Arguing now as in the proof of Proposition \ref{prop:regVhat}, we have that $\widehat{v}_{xx}$ admits a continuous extension to $\overline{\mathcal{C}_3}$, and is therefore bounded therein. Hence, for $y \in \R$, we have that $\widehat{v}_{x}(\cdot,y)$ is Lipschitz continuous on $[c_+(y), c_-(y)]$, with Lipschitz constant $K(y)$ which is locally bounded on $\R$. 
Combining this with the fact that $\widehat{v}_{x}(\cdot,y)$ is infinitely many times continuously differentiable in $\mathcal{S}_3^{\pm}$, thus locally bounded therein, we conclude that $\widehat{v}_{xx}\in L^\infty_{\text{loc}}(\R^2;\R)$.
\end{proof}

\subsection{Integral equations for the free boundaries}
\label{sec:inteq}

By Proposition \ref{prop:reg}, and by using standard arguments based on the strong Markov property (cf.\ \cite{EkstromPes} and \cite{Peskir2008}), we have that the value function $\hat{v}$ defined in \eqref{eq:vhat} and the free boundaries $c_\pm$ satisfy
\begin{align*}
\left\{
\begin{array}{ll}
\big(\mathcal{L}_{X,Y} - \rho\big)\widehat{v}(x,y) = -(1+e^{\frac{\gamma}\eta(x-y)})C'(x), & c_+(y) <  x  < c_{-}(y),\,\,y\in\mathbb{R}\\[+6pt]
\big(\mathcal{L}_{X,Y} - \rho\big)\widehat{v}(x,y) = \rho K^+ (1+e^{\frac{\gamma}\eta(x-y)}), & x < c_+(y),\,\,y\in\mathbb{R}\\[+6pt]
\big(\mathcal{L}_{X,Y} - \rho\big)\widehat{v}(x,y) = - \rho K^- (1+e^{\frac{\gamma}\eta(x-y)}), & x > c_-(y),\,\,y\in\mathbb{R}\\[+6pt]
- K^+ (1+e^{\frac{\gamma}\eta(x-y)}) \leq \widehat{v}(x,y) \leq K^+ (1+e^{\frac{\gamma}\eta(x-y)}), & (x,y)\in\mathbb{R}^2 
\end{array}
\right.
\end{align*}
We recall that $\mathcal{L}_{X,Y}$ is the second-order differential operator defined in \eqref{eq:LXY}, $\widehat{v}\in C^{1}(\R^2;\R)$, $\widehat{v}_{xx}\in L^\infty_{\text{loc}}(\R^2;\R)$ and $\widehat{v} \in C^{2,1}$ inside $\mathcal{C}_3$ (cf.\ Propositions \ref{prop:propcpm}.$(iii)$ and \ref{prop:reg}). 
Hence, via the above results and a suitable application of (a week version of) It\^o's lemma (see, e.g., \cite[Lemma 8.1, Theorem 8.5]{BL} and \cite[Theorem 2.1]{CDeA}), we firstly obtain an integral representation of $\widehat{v}$; since this result is nowadays somehow classical, we omit details.

\begin{proposition}
\label{prop:representingu}
Consider the free boundaries $c_\pm$ defined in \eqref{eq:cpm} and $(X^0,Y^0)$ from \eqref{X0Y0}. Then, for any $(x,y) \in \mathbb{R}^2$, the value function $\widehat{v}$ of \eqref{eq:vhat} can be written as
\begin{align*}
&\widehat{v}(x,y) 
=\E^{\Q}_{(x,y)}\bigg[\int_{0}^{\infty}e^{-\rho s} \big(1+e^{\frac{\gamma}{\eta}(X^0_s - Y^0_s)}\big) C'(X_s^0) \mathds{1}_{\{ c_+(Y^0_s) < X^0_s < c_-(Y^0_s)\}} \d s\bigg] \nonumber \\
&+\E^{\Q}_{(x,y)}\bigg[
\int_{0}^{\infty}e^{-\rho s} \rho \big(1+e^{\frac{\gamma}{\eta}(X^0_s - Y^0_s)}\big) 
\big(K^- \mathds{1}_{\{X^0_s \geq c_-(Y^0_s)\}} - K^+ \mathds{1}_{\{X^0_s \leq c_+(Y^0_s)\}} \big) \d s \bigg] ,
\end{align*}
where $\E^{\Q}_{(x,y)}$ is the expectation under $\Q_{(x,y)}$ such that $(X^0,Y^0)$ starts at $(x,y) \in \R^2$.
\end{proposition}

The previous representation of $\widehat{v}$ allows us to determine a system of integral equations for $c_\pm$ (see \eqref{eq:cpm} for their definition and Proposition \ref{prop:propcpm} for their properties), which is the main aim of this section. 
To this end, denote by $G(z;m,\nu)$ the density function of a Gaussian random variable with mean $m$ and variance $\nu^2$. 
\begin{proposition}
\label{prop:inteq}
Let $q(x,y):= 1 + e^{\frac{\gamma}{\eta}(x-y)}$. The free boundaries $c_\pm$ defined in \eqref{eq:cpm} solve the system of integral equations
\begin{align*}
\mp K^{\pm} q(c_{\pm}(y),y) 
&= \int_{0}^{\infty} e^{-\rho s} \bigg(\int_{\R} q(z, Y^0_s) %\big(1+e^{\frac{\gamma}{\eta}(z - Y^0_s)}\big) %q(X^0_s,Y^0_s) 
\bigg\{ C'(z) \mathds{1}_{\{ c_+(Y^0_s) < z < c_-(Y^0_s)\}} \\
&+ K^- \mathds{1}_{\{z \geq c_-(Y^0_s)\}} - K^+ \mathds{1}_{\{z \leq c_+(Y^0_s)\}} \bigg\} G(z; c_{\pm}(y) + \mu_0 s, \eta^2 s) \d z \bigg) \d s . \nonumber
\end{align*}
Moreover, $(c_+, c_-)$ is the unique solution pair belonging to the set $\mathcal{D}_+ \times \mathcal{D}_-$, where
\begin{align*}
\mathcal{D}_+ &\hspace{-0.1cm}:=\hspace{-0.1cm} \big\{g:\R \to \R:\, \text{$g$ is continuous, nondecreasing, s.t.\ $x^{*}_+ \leq g(y) \leq (C')^{-1}(-\rho K^+)$}\big\} \\
\mathcal{D}_- &\hspace{-0.1cm}:=\hspace{-0.1cm} \big\{g:\R \to \R:\, \text{$g$ is continuous, nondecreasing, s.t.\ $(C')^{-1}(\rho K^-) \leq g(y) \leq x^{*}_-$}\big\}.
\end{align*}
\end{proposition}
\begin{proof}
The integral equations follow by 
taking $x=c_\pm(y)$ in Proposition \ref{prop:representingu}, 
employing the value function's continuity (i.e.\ $\widehat{v}(c_{\pm}(y),y) = \mp K^{\pm} \big(1+\exp\{{\gamma}(c_{\pm}(y) - y)/{\eta}\}\big)$, for any $y \in \R$), and finally noticing that $Y^0$ is a deterministic process and that $X^{0,c_{\pm}(y)}_s$ is Gaussian under $\Q$ with mean $c_{\pm}(y) + \mu_0 s$ and variance $\eta^2 s$. 

The fact that $c_\pm$ belong to the classes $\mathcal{D}_\pm$ follows from their continuity, monotonicity, and boundedness in Proposition \ref{prop:propcpm}. 

Finally, we can proceed as in \cite[Lemmata 3.15, 3.16, Proposition 3.17, Theorem 3.18]{DeAFe} to prove the uniqueness. 
Notice that the problem in \cite{DeAFe} has a finite time-horizon $T$ and the free boundaries satisfy suitable terminal conditions at $T$. However, a careful investigation of the proof of \cite[Lemma 3.15]{DeAFe} reveals that such terminal conditions can be replaced in our problem by the transversality condition (already satisfied by $\widehat{v}$\footnote{Using the relationship \eqref{eq:vhat} between $\hat v$ and $\overline v$ and the definition \eqref{X0Y0} of $(X^0,Y^0)$, we obtain 
\begin{align*}
& \E^{\Q}_{(x,y)}\Big[e^{-\rho T} |\widehat{v}(X^0_{T},Y^0_{T})| \Big] 
= \E^{\Q}_{(x,y)}\Big[e^{-\rho T} \Big|\overline{v}\Big(X^0_{T},e^{\frac{\gamma}{\eta}(X^0_{T} - Y^0_{T})}\Big)\Big| \Big] \nonumber \\
& \leq (K^+  \vee K^-) \E^{\Q}_{(x,\exp\{\frac{\gamma}{\eta}(x-y)\})}\Big[e^{-\rho T} \Big(1 + \Phi_{T}\Big)\Big] = (K^+  \vee K^-)\big(1 + e^{\frac{\gamma}{\eta}(x-y)}\big) e^{-\rho T},
\end{align*}
where the last step is due to the martingale property of the process $\Phi$.})
\begin{equation}
\label{eq:transv-fb}
\lim_{T\uparrow \infty} \E^{\Q}_{(x,y)}\big[ e^{-\rho T}{u}_{\alpha}(X^0_T, Y^0_T)\big]=0,
\end{equation}
imposed on a candidate value function ${u}_{\alpha}$ (cf.\ \cite[Eq.\ (3.56)]{DeAFe}). The arguments in the proofs of \cite[Lemma 3.16, Proposition 3.17, Theorem 3.18]{DeAFe} do not exploit the terminal conditions of the free boundaries, so that they can be adapted to the present setting.
\end{proof}

\begin{remark}
\label{rem:c-and-b}
The complete characterisation of the boundaries $c_\pm$ provided by Proposition \ref{prop:inteq} together with \eqref{Cpm},  yield a complete description of the free boundaries $b_\pm$, at which the optimal control rule $\widehat{P}$ constructed in \eqref{PXPhi0}--\eqref{PXPhi} (see Section \ref{Control} for details) commands the process $(X^{\widehat{P}}_t, \Phi_t)_{t \geq 0}$ to be reflected.

Indeed, once $c_\pm$ are determined by solving (numerically) the system of integral equations in Proposition \ref{prop:inteq}, we can use \eqref{Cpm} to obtain $b^{-1}_{\pm}$, and consequently determine $b_{\pm}$ by inverting \eqref{bpminv}. 
However, such a numerical treatment is non trivial and outside the scopes of the present work, we do not address it in this paper.
\end{remark}

\appendix

%\begin{comment} %%%%%%%%%%%%%% ArXiv
\section{Technical proofs} 

\subsection{Proof of Proposition \ref{prop:semiconc}} 
\label{A}
It follows from \eqref{eq:XPhiP}, that $\Phi_t = \varphi \mathcal{M}_t$, where $\mathcal{M}_t:=\exp\{\gamma W_t - {\gamma^2}t/{2}\}$, for any $t\geq0$ and $\varphi>0$. 

For any $(x,\varphi) \in \R \times (0,\infty)$ given and fixed, we clearly have $\overline{V}(x,\varphi) \leq \overline{\mathcal{J}}_{x,\varphi}(0)$. Hence, without loss of generality, we can restrict the attention to all those controls $P \in \mathcal{A}$ such that, for some constant $\kappa_o>0$ (changing in the rest of this proof), 
\begin{align}
\label{eq:semiconc-4}
\begin{split}
&\E^{\Q}\bigg[\int_0^{\infty} e^{-\rho t} \big(1 + \varphi \mathcal{M}_t \big) C(X^{x;P}_t) \d t \bigg] 
\leq \overline{\mathcal{J}}_{x,\varphi}(P) \leq  \overline{\mathcal{J}}_{x,\varphi}(0) \\
&= \E^{\Q}\bigg[\int_0^{\infty} e^{-\rho t} \big(1 + \varphi \mathcal{M}_t \big) C(X^{x;0}_t) \d t \bigg] 
\leq \kappa_o (1 + \varphi)(1 + |x|^p).
\end{split}
\end{align}
Here, the last inequality follows from a change of measure as in Section \ref{sec:2transf}, Assumption \ref{ass:C}.$(i)$, and standard estimates (recall that $X^{x;0}$ under $\Q$ evolves as in \eqref{XoPhi}, while under $\P$ it evolves as in \eqref{X0Pi}).
We denote this class of controls by $\mathcal{A}_o$.

Then, let $(x,\varphi),(x',\varphi')$ such that $|(x,\varphi)|\leq R$, $|(x',\varphi')|\leq R$ be given and fixed, and take $\lambda \in [0,1]$. 
Observe that, by using the definition \eqref{eq:Vbar} of $\overline{V}$  (and restricting to the class $\mathcal{A}_o$), we get  
\begin{align}
&\lambda  \overline{V}(x,\varphi)+(1-\lambda) \overline{V}(x',\varphi')-\overline{V}(\lambda(x,\varphi)+(1-\lambda)(x',\varphi')) \nonumber \\
&\leq \sup_{P \in \mathcal{A}_o}\E^{\Q}\bigg[\int_0^{\infty} \hspace{-3mm}e^{-\rho t} \Big[\lambda (1 + \varphi \mathcal{M}_t) C(X^{x;P}_t)  + (1-\lambda)(1 + \varphi' \mathcal{M}_t) C(X^{x';P}_t) 
\nonumber \\&
\qquad \qquad \qquad - \big(1 + (\lambda \varphi + (1-\lambda)\varphi')\mathcal{M}_t\big) C(X^{\lambda x + (1-\lambda)x';P}_t)\Big] \d t \nonumber \\
& + \int_0^{\infty} \hspace{-3mm}e^{-\rho t} K^+ \Big[\lambda (1 + \varphi \mathcal{M}_t) + (1-\lambda)(1 + \varphi' \mathcal{M}_t) - \big(1 + (\lambda \varphi + (1-\lambda)\varphi')\mathcal{M}_t\big)\Big] \d P^+_t \nonumber \\
& + \int_0^{\infty} \hspace{-3mm} e^{-\rho t} K^- \Big[\lambda (1 + \varphi \mathcal{M}_t) + (1-\lambda)(1 + \varphi' \mathcal{M}_t) - \big(1 + (\lambda \varphi + (1-\lambda)\varphi')\mathcal{M}_t\big)\Big] \d P^-_t \bigg]. \nonumber
\end{align}
By adding and subtracting 
$(1-\lambda) \varphi \mathcal{M} \big( C(X^{x';P}) + C(\lambda X^{x;P}_t + (1-\lambda)X^{x';P}_t) \big)$ in the $\d t$-integral appearing in the last equation, using the semiconcavity property of $C$ in Assumption \ref{ass:C}.$(iii)$ together with the solution $X^{x;P}$ of \eqref{eq:XPhiP}, as well as the fact that $\sup(f + g) \leq \sup(f) + \sup(g)$, we obtain
\begin{align}
& \lambda  \overline{V}(x,\varphi)+(1-\lambda) \overline{V}(x',\varphi')-\overline{V}(\lambda(x,\varphi)+(1-\lambda)(x',\varphi')) \nonumber \\ 
& \leq \sup_{P \in \mathcal{A}_o} \hspace{-1.5mm}\E^{\Q}\bigg[\int_0^{\infty} \hspace{-3mm}e^{-\rho t} \alpha_2\lambda(1-\lambda)\Big(1 + C(X^{x;P}_t) + C(X^{x';P}_t)\Big)^{(1-\frac{2}{p})}|x-x'|^2 \d t\bigg] \nonumber \\
& + \hspace{-1.5mm}\sup_{P \in \mathcal{A}_o}\hspace{-1.5mm}\E^{\Q}\bigg[\int_0^{\infty} \hspace{-3mm}e^{-\rho t} \varphi \mathcal{M}_t \big(\lambda C(X^{x;P}_t) + (1-\lambda)C(X^{x';P}_t) - C(\lambda X^{x;P}_t + (1-\lambda)X^{x';P}_t)\big) \d t \nonumber \\
& \qquad \qquad + \int_0^{\infty} \hspace{-3mm}e^{-\rho t} (1-\lambda) (\varphi - \varphi') \mathcal{M}_t \big( C(\lambda X^{x;P}_t + (1-\lambda)X^{x';P}_t) - C(X^{x';P}_t)\big) \d t\bigg]. \nonumber
\end{align}
Using again the assumed semiconcavity of $C$ and H\"older's inequality, we further conclude that
\begin{align*}
& \lambda  \overline{V}(x,\varphi)+(1-\lambda) \overline{V}(x',\varphi')-\overline{V}(\lambda(x,\varphi)+(1-\lambda)(x',\varphi'))  
\leq \alpha_2 \lambda(1-\lambda) 
|x-x'|^2 \\ 
&\times \bigg(
\E^{\Q}\bigg[\int_0^{\infty} \hspace{-3mm}e^{-\rho t} \d t\bigg]^{\frac{2}{p}} \hspace{-1.5mm}\sup_{P \in \mathcal{A}_o} \hspace{-1.5mm}\E^{\Q}\bigg[\int_0^{\infty} \hspace{-3mm}e^{-\rho t} \Big(1 + C(X^{x;P}_t) + C(X^{x';P}_t)\Big) \d t\bigg]^{(1-\frac{2}{p})} \\
&+\hspace{-1.5mm} \sup_{P \in \mathcal{A}_o}\hspace{-1.5mm}\E^{\Q}\bigg[\int_0^{\infty} \hspace{-3mm}e^{-\rho t} \varphi \mathcal{M}_t \Big(1 + C(X^{x;P}_t) + C(X^{x';P}_t)\Big)^{(1-\frac{2}{p})} \hspace{-2.5mm} \d t\bigg]\bigg)  
+ \alpha_1 \lambda(1-\lambda)|\varphi - \varphi'| \\
&\times|x-x'| \hspace{-1.5mm} \sup_{P \in \mathcal{A}_o}\hspace{-1.5mm}\E^{\Q}\bigg[\int_0^{\infty} \hspace{-3mm}e^{-\rho t} \mathcal{M}_t \Big(1 + C(\lambda X^{x;P}_t + (1-\lambda)X^{x';P}_t) + C(X^{x';P}_t)\Big)^{(1-\frac{1}{p})} \hspace{-1mm} \d t \bigg]. \nonumber
\end{align*}
Then, using $\E^{\Q}[\int_0^{\infty} e^{-\rho t} \mathcal{M}_t \d t] = 1/\rho$ and H\"older's inequality, we further obtain that
\begin{align}
&\lambda  \overline{V}(x,\varphi)+(1-\lambda) \overline{V}(x',\varphi')-\overline{V}(\lambda(x,\varphi)+(1-\lambda)(x',\varphi')) \nonumber \\ 
&\leq \alpha_2 \lambda(1-\lambda) |x-x'|^2 \bigg(\rho^{-\frac{2}{p}}\Big(\rho^{-1} + \kappa_o(1+|x|^p+|x'|^p)\Big)^{1-\frac{2}{p}} + \E^{\Q}\bigg[\int_0^{\infty} \hspace{-3mm}e^{-\rho t} \varphi\mathcal{M}_t \d t\bigg]^{\frac{2}{p}} \nonumber \\
&\times \hspace{-1mm}\E^{\Q}\bigg[\int_0^{\infty} \hspace{-3mm}e^{-\rho t} \varphi\mathcal{M}_t\Big(1 + C(X^{x;P}_t) + C(X^{x';P}_t)\Big) \d t\bigg]^{1-\frac{2}{p}} \bigg) 
+ \alpha_1 \lambda(1-\lambda) |\varphi - \varphi'||x-x'| \nonumber \\ 
&\times \hspace{-2mm}\sup_{P \in \mathcal{A}_o} \hspace{-1.5mm}\E^{\Q}\bigg[\int_0^{\infty} \hspace{-3mm}e^{-\rho t} \mathcal{M}_t \Big(1 \hspace{-1mm}+ \hspace{-1mm}C(\lambda X^{x;P}_t \hspace{-1mm}+ \hspace{-1mm}(1-\lambda)X^{x';P}_t) \hspace{-1mm}+ \hspace{-1mm}C(X^{x';P}_t)\Big) \d t \bigg]^{1-\frac{1}{p}} \hspace{-4.5mm}\E^{\Q}\bigg[\int_0^{\infty} \hspace{-3mm}e^{-\rho t} \mathcal{M}_t \d t \bigg]^{\frac{1}{p}} \hspace{-2mm}. \nonumber
\end{align}
Hence, employing estimate \eqref{eq:semiconc-4} in the above inequality, we find for some $\kappa>0$ that 
\begin{align}
& \lambda  \overline{V}(x,\varphi)+(1-\lambda) \overline{V}(x',\varphi')-\overline{V}(\lambda(x,\varphi)+(1-\lambda)(x',\varphi')) \nonumber \\ 
& \leq \kappa \lambda(1-\lambda) \Big( |x-x'|^2 (1 + \varphi)(1 + |x| + |x'|)^{p-2} + |\varphi - \varphi'||x-x'| (1 + |x| + |x'|)^{p-1} \Big) \nonumber 
\end{align}
which gives the claimed local semiconcavity.
\hfill \cvd

%\end{comment} %%%%%%%%%%%%%%

%%%%%%%%%%
%
%%\vspace{-0.02cm}
%%\subsection
%{\bf Proof of Lemma \ref{c-1}.}
%%
%Due to \eqref{bpminv} and Proposition \ref{prop:barb}, we have that $b^{-1}_\pm$ are nonincreasing, with $b^{-1}_+$ left-continuous and $b^{-1}_-$ right-continuous. Combining the aforementioned properties together with the definition \eqref{Cpm} yields the desired properties.
%\hfill \cvd
%%%%%%%%%%

%\begin{comment} %%%%%%%%%%%%%%

\subsection{Proof of Proposition \ref{prop:sf}}
\label{B}

We focus on proving the continuity of $\hat{v}_x$ across $c_+$, since the other claims can be obtained similarly. 
To that end, we firstly simplify the notation by defining (cf.\ \eqref{C3hat}--\eqref{S3+}) $q(x,y):=1+e^{\frac{\gamma}{\eta}(x-y)}$ and
\begin{equation*}
\hat{w}(x,y):=\widehat{v}(x,y) +K^+q(x,y) 
\begin{cases} > 0 , & \text{for all } (x,y) \in \R^2 \setminus {\mathcal{S}_3^+}, \\
= 0 , & \text{for all } (x,y) \in \mathcal{S}^+_3.
\end{cases}
\end{equation*}
Notice that, for every $(x,y)\in \R^2$ we have 
\begin{multline*}
\hat{w}(x,y) 
=\sup_{\tau\in\mathcal{T}}\inf_{\sigma\in\mathcal{T}}
\E^{\Q}\bigg[\int_0^{\tau\wedge\sigma} e^{-\rho t}q(X^{0,x}_t,Y^{0,y}_t)\big(C'(X^{0,x}_t)+\rho K^+\big)\d t \\ 
+(K^+ + K^-)e^{-\rho \sigma}q(X^{0,x}_\sigma,Y^{0,y}_\sigma) \mathbf{1}_{\{\sigma<\tau\}}\bigg].
\end{multline*}

Then, the desired continuity of $\hat{v}_x$ across $c_+$ is equivalent to 
\begin{equation}
\label{lim}
\lim_{\mathcal{C}_3 \ni (x,y)\to(x_o,y_o)} \hat{w}_{x}(x,y)=0, \quad \text{for}\quad x_o:=c_+(y_o) \quad \text{and} \quad y_o \in \R.
\end{equation}

To prove this, fix $(x,y)\in\mathcal{C}_3$ and let $\varepsilon>0$ be such that $(x+\varepsilon,y)\in\mathcal{C}_3$.
Denote by $\tau^{\star}\equiv \tau^{\star}(x,y)$ and $\check{\tau}^{\star}\equiv \check{\tau}^{\star}(x,y)$ from \eqref{star} and \eqref{check}, respectively. 
Then, define 
$\tau^{\star}_\varepsilon:=\tau^{\star}(x+\varepsilon,y)$ according to \eqref{star}
and $\check{\tau}^{\star}_\varepsilon:=\check{\tau}^{\star}(x+\varepsilon,y)$ according to \eqref{check}. 
In view of Proposition \ref{prop:propcpm}.$(iii)$, these take the form
\begin{align*}
\tau^{\star}_\varepsilon &= \inf\{t\geq 0: \ X^{0,x+\varepsilon}_t \leq c_+(Y^{0,y}_t)\} , \quad
\check{\tau}^{\star}_\varepsilon=\inf\{t\geq 0: \ X^{0,x+\varepsilon}_t< c_+(Y^{0,y}_t)\} ,
\\
\tau^{\star} &= \inf\{t\geq 0: \ X^{0,x}_t \leq c_+(Y^{0,y}_t)\}
\quad \text{and} \quad 
\check{\tau}^{\star}=\inf\{t\geq 0: \ X^{0,x}_t< c_+(Y^{0,y}_t)\}.
\end{align*}
By the regularity of the Brownian motion, we have 
$\tau^{\star}_\varepsilon=\check{\tau}^{\star}_\varepsilon$ and $\tau^{\star}=\check{\tau}^{\star}$,
and by the continuity of trajectories of the Brownian motion, we have 
\begin{equation}
\label{tau+}
\lim_{\varepsilon\downarrow 0} \check{\tau}^{\star}_\varepsilon \to \check{\tau}^{\star}
\quad \text{which eventually yields that} \quad 
\lim_{\varepsilon\downarrow 0} \tau^{\star}_\varepsilon \to \tau^{\star}.
\end{equation} 
Moreover, Proposition \ref{prop:propcpm}.$(iii)$ further implies that $\sigma^{\star} \equiv \sigma^{\star}(x,y)$ from \eqref{star} takes the form
$\sigma^{\star} =\inf\{t\geq 0: \ X^{0,x}_t\geq c_-(Y^{0,y}_t)\}.$
Then, using the Mean-Value Theorem,  
\begin{align*} 
&\frac{\hat{w}(x+\varepsilon,y) -\hat{w}(x,y)}{\varepsilon} 
\leq \E^{\Q}\bigg[\int_0^{\tau^\star_\varepsilon\wedge\sigma^\star} \hspace{-5pt}e^{-\rho t}
q_x(\Lambda^\varepsilon_t,Y^{0,y}_t)\big(C'(X^{0,x+\varepsilon}_t)+\rho K^+\big)\d t 
\bigg] \nonumber 
\\
&+\E^{\Q}\bigg[\int_0^{\tau^\star_\varepsilon\wedge\sigma^\star} 
\hspace{-5pt}e^{-\rho t}q(X^{0,x}_t,Y^{0,y}_t)C''(\Xi^\varepsilon_t)\d t + e^{-\rho \sigma^\star}\mathbf{1}_{\{\tau^\star_\varepsilon>\sigma^\star\}}(K^+ + K^-)q_x(\Theta^\varepsilon_{\sigma^\star})\bigg].
\end{align*}
where $\Lambda_t^\varepsilon, \Xi^\varepsilon_t\in (X^{0,x}_t, X^{0,x+\varepsilon}_t)$ and $\Theta^\varepsilon_{\sigma^\star}\in (X^{0,x}_{\sigma^\star}, X^{0,x+\varepsilon}_{\sigma^\star})$. 
Thus, using the dominated convergence theorem and also \eqref{tau+}, 
\begin{align*}
&\limsup_{\varepsilon\downarrow 0} \,\frac{\hat{w}(x+\varepsilon,y)-\hat{w}(x,y)}{\varepsilon}
\leq \E^{\Q}\bigg[\int_0^{\tau^\star\wedge\sigma^\star} \hspace{-3mm} e^{-\rho t}
q_x(X^{0,x}_t,Y^{0,y}_t)\big(C'(X_t^{0,x})+\rho K^+\big)\d t 
\bigg]
\\
&+\,\E^{\Q}\bigg[\int_0^{\tau^\star\wedge\sigma^\star} \hspace{-3mm}e^{-\rho t}q(X_t^{0,x},Y^{0,y}_t)C''(X^{0,x}_t)\d t 
+ e^{-\rho \sigma^\star}\mathbf{1}_{\{\tau^\star\geq \sigma^\star\}}(K^+ + K^-)q_x(X^{0,x}_{\sigma^\star})\bigg].
\end{align*}
With similar estimates, we can also prove the opposite inequality for the $\liminf$ as ${\varepsilon\downarrow 0}$.
This allows us to conclude that 
\begin{align*}
&\hat{w}_x(x,y) 
= \E^{\Q}\bigg[\int_0^{\tau^\star\wedge\sigma^\star} \hspace{-3mm}e^{-\rho t}
q_x(X^{0,x}_t,Y^{0,y}_t)\big(C'(X_t^{0,x})+\rho K^+\big)\d t 
\bigg]
\\
&+\,\E^{\Q}\bigg[\int_0^{\tau^\star\wedge\sigma^\star} \hspace{-3mm}e^{-\rho t}q(X_t^{0,x},Y^{0,y}_t)C''(X^{0,x}_t)\d t 
+ e^{-\rho \sigma^\star}\mathbf{1}_{\{\tau^\star\geq \sigma^\star\}}(K^+ + K^-)q_x(X^{0,x}_{\sigma^\star})\bigg].
\end{align*}
Then, we obtain \eqref{lim} by taking the limit as $(x,y)\to(x_0,y_0)$, using  Lemma \ref{lemmasf} and noticing that clearly $\liminf_{(x,y)\to (x_0,y_0)}\sigma^\star(x,y)>0$ (cf.\ \eqref{S3+}). 

In order to complete the proof, it remains to show that the dominated convergence theorem can be indeed invoked when taking limits above. 
We show this only for
$$
\E^{\Q}\bigg[\int_0^{\tau^\star_\varepsilon\wedge\sigma^\star}e^{-\rho t} q_x(\Lambda^\varepsilon_t,Y^{0,y}_t)\big(C'(X^{0,x+\varepsilon}_t)+\rho K^+\big)\d t 
\bigg]
$$
as other terms can be treated similarly. 
Notice that, since $q_x(\cdot,y)$ is positive and increasing, $C'(\cdot)$ is nondecreasing, and $\Lambda^{\varepsilon}_t \leq X^{0,x+\varepsilon}_t \leq X^{0,x+1}_t$ (for any $\varepsilon < 1$, without loss of generality), we can write
\begin{align*}
& \int_0^{\tau^\star_\varepsilon\wedge\sigma^\star} \hspace{-3mm}e^{-\rho t} q_x(\Lambda^\varepsilon_t,Y^{0,y}_t)\big(C'(X^{0,x+\varepsilon}_t)+\rho K^+\big)\d t \nonumber \\& 
\leq \frac{\gamma}{\eta} \int_0^{\infty}e^{-\rho t} \big(q(X^{0,x+1}_t,Y^{0,y}_t) + 1) \big(\big|C'(X^{0,x+1}_t)\big|+\rho K^+\big)\d t \nonumber 
\end{align*}
Now, on one hand, $\E^{\Q}[\int_0^{\infty}e^{-\rho t} |C'(X^{0,x+1}_t)| \d t]<\infty$ due to Assumption \ref{ass:C} and standard estimates on the Brownian motion. On the other hand, by using the definition of $q(\cdot,\cdot)$, 
one has $q(X^{0,x+1}_t,Y^{0,y}_t) = 1 + \Phi^{\varphi}_t$, with $\varphi \equiv e^{\frac{\gamma}{\eta}(x+1-y)}$. Hence,
\begin{align}
&\E^{\Q}\bigg[\int_0^{\infty}e^{-\rho t} q(X^{0,x+1}_t,Y^{0,y}_t)\left(\big|C'(X^{0,x+1}_t)\big|+\rho K^+\right)\d t\bigg] 
\nonumber \\& 
= \Big(1 + e^{\frac{\gamma}{\eta}(x+1-y)}\Big) \E\bigg[\int_0^{\infty}e^{-\rho t}\big(\big|C'(X^{0,x+1}_t)\big| + \rho K^+\big)\d t\bigg], 
\notag
\end{align}
where the last equality is due to a change of measure as in Section \ref{sec:2transf} and $X^0$ in the last expectation evolves as in \eqref{X0Pi}. 
But then, standard estimates together with the growth requirements on $C$ in Assumption \ref{ass:C} ensure that the last expectation in the above formula is finite, thus completing the proof.
\hfill \cvd

\subsection{A Technical Result}
%\end{comment} %%%%%%%%%%%%%%
%%%%%%%%%%%%%%
%\section{Technical Result}
%%%%%%%%%%%%%%
\label{C}

\begin{lemma}
\label{lemma-app}
Let $W$ be a one-dimensional Brownian motion on the complete filtered probability space $(\Omega, \mathcal{F}, \mathbb{F}, \Q)$,  $\{\tau_{k}\}_{k\geq1}$ be a strictly increasing sequence of $\mathbb{F}$-stopping times diverging a.s., $\zeta, \beta, c>0$, $\alpha \in \mathbb{R}$, $f:\R \to \R$ be nonincreasing, and $g:\R \to \R$ be Lipschitz-continuous. Then, for each $t>0$,
\begin{eqnarray*}
&\Q\left(\bigcup_{k=1}^{\infty} \big\{t\in (\tau_{k-1},\tau_{k}]\big\}\cap \big\{t\in \arg\!\max_{s\in[\tau_{k-1},t]}(f(c e^{\alpha s+\beta W_{s}}) - \zeta W_{s}+g(s))\big\}\right)=0 \nonumber \\
&\Q\left(\bigcup_{k=1}^{\infty} \big\{t\in (\tau_{k-1},\tau_{k}]\big\}\cap \big\{t\in \arg\!\min_{s\in[\tau_{k-1},t]}(f(c e^{\alpha s+\beta W_{s}}) - \zeta W_{s}+g(s))\big\}\right)=0. \nonumber
\end{eqnarray*}
\end{lemma}

\begin{proof} We show the claim only for the argmax.
Fix $t>0$ and set  $\Omega_{k}:=\big\{t\in (\tau_{k-1},\tau_{k}]\big\}$. 
The proof can be concluded by showing that for each $k\geq 1$, 
$$
\Q\big(t \in \argmax_{s\in[\tau_{k-1},t]}(f(c e^{\alpha s+\beta W_{s}}) - \zeta W_{s}+g(s)) \ | \ \Omega_{k}\big)=0 .
$$
With a change of measure, the above is equivalent to 
\begin{eqnarray*}
&\widehat{\Q}\big(t\in \argmax_{s\in[\tau_{k-1},t]}(f(c e^{\beta W^{*}_{s}}) - \zeta 
W^{*}_{s}+h(s))\ | \ \Omega_{k} \big)=0 \nonumber,
\end{eqnarray*}
for another $\mathbb{F}$-Brownian motion $W^*$ and $h:\R \to \R$ Lipschitz-continuous. 
Now, for each $\tau_{k-1}< s\leq t$,  we have
$$
\big(f(c e^{\beta W^{*}_{t}}) - \zeta W^{*}_{t}\big) - \big(f(c e^{\beta W^{*}_{s}}) - \zeta W^{*}_{s}\big)\leq -\zeta(W^*_{t}-W^{*}_{s}), \ \ \mbox{if} \  \ W^*_{t}-W^{*}_{s}\geq 0.
$$
By the path-properties of the Brownian motion, we have $\widehat{\Q}\big(\cdot\,|\, \Omega_{k}\big)$-a.s.
$$
\limsup_{s\to t^{-}} \tfrac{W^{*}_{t} - W^{*}_{s}}{t-s}=+\infty.
$$
In particular, $\widehat{\Q}\big(\cdot\,|\, \Omega_{k}\big)$-a.s., there exists a sequence $s_{n}\to t^{-}$ (possibly depending on $\omega$) such that  
$$
 W^*_{t}-W^{*}_{s_{n}}\geq 0 \;\;\; \forall n \ \ \mbox{and}  \ \ \limsup_{n\to \infty}\tfrac{W^{*}_{t} - W^{*}_{s_{n}}}{t-s_{n}}=+\infty.
$$
Hence, the claim follows by observing that, $\widehat{\Q}\big(\cdot\,|\, \Omega_{k}\big)$-a.s., we have 
\begin{align*}
&\liminf_{s\to t^{-}}\frac{1}{t-s}\big[\big(f(ce^{\beta W^{*}_{t}}) - \zeta W^{*}_{t}+h(t)\big) - \big(f(ce^{\beta W^{*}_{s}}) - \zeta W^{*}_{s}+h(s)\big)\big]\\
&\leq \liminf_{n\to\infty}\frac{1}{t-s_{n}}\big[\big(f(ce^{\beta W^{*}_{t}}) - \zeta W^{*}_{t}+h(t)\big) - \big(f(ce^{\beta W^{*}_{s_{n}}}) - \zeta W^{*}_{s_{n}}+h(s_{n})\big)\big]\\
&\leq   \liminf_{n\to \infty}\Big(- \zeta \tfrac{W^{*}_{t} - W^{*}_{s_{n}}}{t-s_{n}} \Big)+\limsup_{n\to\infty} \,\tfrac{|h(t)-h(s_{n})|}{t-s_{n}}\\
&=  -\zeta \limsup_{n\to \infty} \,\tfrac{W^{*}_{t} - W^{*}_{s_{n}}}{t-s_{n}}+\limsup_{n\to\infty} \,\tfrac{|h(t)-h(s_{n})|}{t-s_{n}}= -\infty. 
\end{align*}
%and the claim follows.
\end{proof}

%\medskip

%\appendix
%\section{An example appendix} 
%\lipsum[71]

\indent \textbf{Acknowledgments.} We thank the anonymous Associate Editor and Referee for constructive comments.

%\section*{Acknowledgments}
%Financial support by the German Research Foundation (DFG) through the Collaborative Research Centre 1283 is gratefully acknowledged by the second author. 

%%%%%%%%%%%%%%%%%%%%%%%%%%%%%%%%%%%%%%%%%%%%%%%%%%%%%%%%%%%%%%%%%%%%%%%%%%

\bibliographystyle{siamplain}
%\bibliography{references}

\end{document}